\documentclass[12pt]{amsart}%{article}
\usepackage{amsmath,amstext,amsfonts,amsthm,amssymb}
\usepackage{eucal}
\usepackage{diagrams}                                      %
\diagramstyle[scriptlabels,PostScript=dvips]               %
\newarrow{Dash}{}{dash}{}{dash}{}                          %
\newarrow{Equal}=====
\mathchardef\mhyphen="2D
\swapnumbers
\newtheorem{thm}[subsubsection]{Theorem}
\newtheorem{lem}[subsubsection]{Lemma}
\newtheorem{prp}[subsubsection]{Proposition}
\newtheorem{crl}[subsubsection]{Corollary}

{  \theoremstyle{definition}
           \newtheorem{dfn}[subsubsection]{Definition}
           \newtheorem{exm}[subsubsection]{Example}
           \newtheorem{rem}[subsubsection]{Remark}

           \newtheorem*{Dfn}{Definition}

}

\newcommand{\AW}{\bigtriangleup} %Alexander_Whitney map
\newcommand{\dg}{\mathtt{dg}}

\newcommand{\Cat}{\mathtt{Cat}}
\newcommand{\st}{\mathit{st}}

\newcommand{\CM}{\mathtt{CM}}
\newcommand{\Coc}{\mathtt{Coc}}

\newcommand{\colim}{\operatorname{colim}}
\newcommand{\Com}{\mathtt{Com}}
\newcommand{\Cone}{\operatorname{Cone}}

\newcommand{\EM}{\mathtt{\bigtriangledown}} %Eilenberg-MacLane map
\newcommand{\Fin}{\mathit{Fin}}
\newcommand{\Fun}{\operatorname{Fun}}
\newcommand{\Free}{\mathfrak{F}}

\newcommand{\Hom}{\mathrm{Hom}} % straight Hom

\newcommand{\chom}{\mathcal{H}\mathit{om}} % calligraphic Hom for complexes
\newcommand{\id}{\mathrm{id}}

\newcommand{\iso}{\mathrm{iso}}

\newcommand{\lax}{\mathrm{lax}}
\newcommand{\Left}{\mathbf{L}}
\newcommand{\LEFT}{\mathtt{Left}} % for the category of left fibrations
\newcommand{\Mod}{\mathtt{Mod}}
\newcommand{\Map}{\operatorname{Map}}

\newcommand{\one}{\mathbf{1}}
\newcommand{\Ob}{\operatorname{Ob}}
\newcommand{\op}{\mathrm{op}}

\newcommand{\Pair}{\mathtt{Pair}}

\newcommand{\Pop}{\mathrm{Pop}}

\newcommand{\QC}{\mathtt{QC}}
\newcommand{\rlarrows}{\stackrel{\rTo}{\lTo}}

\newcommand{\SM}{\mathrm{SM}}

\newcommand{\Set}{\mathtt{Set}}
\newcommand{\sCat}{\mathtt{sCat}}
\newcommand{\sMod}{\mathtt{sMod}}
\newcommand{\sSet}{\mathtt{sSet}}

\newcommand{\Sym}{\operatorname{Sym}}

\newcommand{\triv}{\mathrm{triv}}

\newcommand{\wt}{\widetilde}

\newcommand{\Alg}{\mathtt{Alg}}

\newcommand{\cC}{\mathcal{C}}
\newcommand{\cD}{\mathcal{D}}

\newcommand{\cI}{\mathcal{I}}
\newcommand{\cJ}{\mathcal{J}}
\newcommand{\cK}{\mathcal{K}}
\newcommand{\cL}{\mathcal{L}}
\newcommand{\cM}{\mathcal{M}}

\newcommand{\cO}{\mathcal{O}}
\newcommand{\cP}{\mathcal{P}}
\newcommand{\cQ}{\mathcal{Q}}
\newcommand{\cR}{\mathcal{R}}
\newcommand{\cS}{\mathcal{S}}

\newcommand{\fC}{\mathfrak{C}}

\newcommand{\fN}{\mathfrak{N}}

\newcommand{\F}{\mathbb{F}}
\newcommand{\G}{\mathbb{G}}

\newcommand{\Q}{\mathbb{Q}}

\newcommand{\Mu}{\mathtt{M}}

\begin{document}

\title[]{Rectification of algebras and modules}
\author{Vladimir Hinich}
\address{Department of Mathematics, University of Haifa,
Mount Carmel, Haifa 3498838,  Israel}
\email{hinich@math.haifa.ac.il}
\begin{abstract}
Let $\cO$ be a topological (colored) operad. The Lurie $\infty$-category
of $\cO$-algebras with values in ($\infty$-category of) complexes is compared
to the $\infty$-category underlying the model category of (classical) 
dg $\cO$-algebras. This can be interpreted as a "rectification" result
for Lurie operad algebras. A similar result is obtained for modules
over operad algebras, as well as for algebras over topological PROPs.
\end{abstract}
\maketitle
\section{Introduction}

\subsection{Operad algebras}
In this paper we compare two notions of operad algebras with values in complexes.
Let $\cO$ be a topological colored symmetric operad. The functor
of singular chains with coefficients in a commutative ring $k$ 
converts $\cO$ into an operad in the category of complexes,
so that one has the category $\Alg_\cO(C(k))$ of
$\cO$-algebras in $C(k)$ in the ``conventional" sense: its objects are
complexes $A\in C(k)$ together with $\Sigma_n$-equivariant operations
$$ C_*(\cO(n),k)\otimes A^{\otimes n}\rTo A$$
satisfying the standard compatibilities.

The category $\Alg_\cO(C(k))$ has (sometimes) a 
model category structure with quasiisomorphisms as weak equivalences
 and surjective maps as fibrations. Sometimes it does not have such
 model structure. In any case, one can find a quasiisomorphism of dg
 operads $\cR\to C_*(\cO)$ such that the category of $\cR$-algebras has 
a model structure; moreover, under a mild extra requirement, 
the model category $\Alg_\cR(C(k))$ is independent, up to Quillen
 equivalence, of the dg operad $\cR$. The operad $\cR$ satisfying the above
 properties will be called {\sl homotopically sound},
see~\ref{dfn:hsound} below.
 
\

A topological operad $\cO$ defines, on the other 
hand, an $\infty$-operad $\cO^\otimes$ in the sense of 
Lurie, \cite{L.HA}, Section 2, and an 
$\infty$-category of algebras $\Alg_{\cO}(\QC(k))$ with values in
the symmetric monoidal
$\infty$-category $\QC(k)$ which is the $\infty$-category version 
of the derived 
category of $k$-modules. Our main result Theorem~\ref{thm:rect-alg} 
claims that, given a quasiisomorphism of operads $\cR\to C_*(\cO)$ 
with $\cR$ homotopically sound\footnote{any operad is homotopically
sound if (and only if) $k$ is a field of characteristic zero.}
 the $\infty$-category $\Alg_\cO(\QC(k))$ is equivalent to the 
$\infty$-category underlying the ``classical" model category $\Alg_\cR(C(k))$.
This can be interpreted as a {\sl rectification} result: any Lurie $\cO$-algebra with values in $\QC(k)$ can be presented by a 
strict $\cR$-algebra. In good cases, when $C_*(\cO)$ is homotopically sound, any Lurie $\cO$-algebra can be presented by a strict $\cO$-algebra with values in $C(k)$.

We feel this is an important (though not unexpected) result: 
the notion of operad algebra in Lurie theory is very flexible; however, 
an algebra over an $\infty$-operad is defined by a huge collection 
of coherence data which is difficult to specify. The description of
$\Alg_\cO(\QC(k))$ as a nerve of a model category allows one to
 present $\cO^\otimes$-algebras in $\QC(k)$ and 
their diagrams by strict $\cO$-algebras in complexes.

\subsubsection{}
The mere formulation of Theorem~\ref{thm:rect-alg} requires a model
category structure on the category of algebras over a colored dg operad. An account of the relevant theory is presented in Section~\ref{sec:models-a}. The results of this section are mostly
well-known, at least for colorless operads. 
Our approach here is
very close to the earlier colorless version \cite{haha}. 

Note that there exists a very general result by C.~Berger and I.~Moerdijk, \cite{BM}, on model structure for algebras over color operads. Unfortunately, we were unable to deduce from their
result that all dg operads are admissible in case $k\supset\Q$.
This is why we felt it important to present the colored version 
of the notion of $\Sigma$-splitness used in \cite{haha}.

%\fbox{Rewrite the following sentence}

%We also show in Section~\ref{sec:models-a} that the forgetful functor
%preserves cofibrations of cofibrant objects in $\Alg_R(C(k))$ ---
%this is an important observation for the rectification result
%of Section~\ref{sec:rectification}. 

\subsubsection{Dold-Kan} 

As we mentioned above, a simplicial operad can be converted, via the normalized chains functor, into a dg operad. This is due to the fact 
that the normalized chains functor
$$ C_*:\sSet\rTo C(k)$$
is lax symmetric monoidal, via Eilenberg-MacLane map.

Were it really symmetric monoidal, any strict 
operad algebra over $\cO$ with values in $C(k)$ would automatically
define an $\cO$-algebra in the sense of Lurie. In real life this is
``almost so" --- the functor $C_*$ induces an adjoint
pair between the symmetric monoidal $\infty$-categories
$\Cat_\infty$ and $\dg\Cat$.

This ``almost so" has to be explained, and we do so in Section~\ref{s:SM} which precedes the rectification theorem.

\subsubsection{}
The proof of Theorem~\ref{thm:rect-alg}
follows the idea of Lurie's Theorem 4.1.4.4 of \cite{L.HA} where a similar result for associative algebras with values in a combinatorial monoidal model category is proven.

\subsection{Algebras over PROPs}

Theorem~\ref{thm:rect-alg} allows one to (partially) rectify
algebras more general than algebras over operads, such as, for instance, 
associative bialgebras.

These are algebras over PROPs, that is, symmetric monoidal functors from a 
certain symmetric monoidal category designed to describe the necessary 
structure (PROP), to the category of complexes.

A topological PROP $P$ gives rise to a SM $\infty$-category $P^\otimes$; this leads to
the notion of $P$-algebra with values in complexes as a SM functor 
$P^\otimes\to \QC(k)^\otimes$.

We do not expect such algebras to be always presentable by strict $P$-algebras in
complexes. One can, however, slightly generalize the
notion of strict dg algebra over a PROP --- allowing lax symmetric monoidal functors
to complexes which are ``homotopy SM", see Definition \ref{dfn:prop-alg}\footnote{A very close notion
was used by K.~Costello in~\cite{C} (h-split symmetric monoidal functors).}. 

Our rectification theorem~\ref{thm:rect-alg} implies easily the equivalence
of two approaches, see Corollary~\ref{crl:rect-prop-alg}.

\subsection{Modules}

The notion of module over an operad algebra is very straightforward in the 
``classical" theory. We describe a construction which assigns to a 
topological operad $\cO$ another operad $\Mu\cO$ whose algebras are pairs 
$(A,M)$ where $A$ is an $\cO$-algebra and $M$ is an $A$-module.
A similar construction can be easily defined in the world of $\infty$-operads 
as well. 

J.~Lurie suggests another notion of module over an operad algebra. 
Similarly to the cases of commutative or associative algebras where modules
(or bimodules) form a symmetric monoidal (or simply monoidal) category, 
his version of the $\infty$-category of modules over a fixed $\infty$-operad 
algebra $A$ has an $\cO$-monoidal structure. To have such nicely behaved 
notion, one has to require some very special properties from $\cO$ --- it has to
be {\em coherent}, see \cite{L.HA}, Sect.~3.

In Appendix \ref{app:2mod} we prove that our definition of module coincides with the one suggested by Lurie (with discarded $\cO$-monoidal structure). Our 
rectification result easily implies the rectification for modules, see 
Corollary~\ref{crl:str-mod}.

\subsection{SM adjunction}  Appendix \ref{sec:SMA}
deals with adjunction between two (or more) symmetric monoidal (in what follows: SM)
$\infty$-categories. It turns out that if any of the adjunctions 
is a lax SM functor, all others automatically acquire the same structure. This generalizes a well-known observation   for
conventional SM categories that the functor right
adjoint to a SM functor, is necessarily lax. We believe that this
is an important fact in its own. In this paper we use it to construct
the functor from strict operad algebras to algebras over the corresponding $\infty$-operad.

Appendix~\ref{app:2mod} contains some technical details of the comparison between two notions of module.

\subsection{Acknowledgements} Parts of this paper were written during author's visit to MIT and IHES. I am grateful to these institutions for hospitality and excellent 
working conditions. I am very grateful to J.~Lurie for several useful conversations. I am also very grateful to the referee for numerous suggestions, as well as to F.~Muro for pointing out to a wrong formula in the earler version of the manuscript.

\section{Models for algebras}\label{sec:models-a}

\subsection{Introduction}

In this section we present an account of the 
model category structure on algebras over colored operads.
The results described in this section are mostly well-known,
for the colorless case see \cite{haha,virtual,BM0}, and 
for the colored case \cite{BM} (Berger and Moerdijk consider more general algebras with values in any model category).

Having in mind applications to dg algebras, we wanted to make sure
the important case of algebras in characteristic zero would be covered.
The only proof we are aware of in the colorless case is via the
notion of $\Sigma$-split operads presented in \cite{haha}. This is why
we spend some time to give a colored version of this notion.
 
We also use this section as an opportunity to fix notation for 
colored operads.

\subsection{Colored operads}

Fix a symmetric monoidal category $\cC$. A colored operad $\cO$ in $\cC$  consists of the following data.
\begin{itemize}
\item[1.] A set $[\cO]$ (the set of colors of $\cO$)
\footnote{One can allow $[\cO]$ to be ``big'', as the set of objects of a category.}.
\item[2.] An object $\cO(c,d)$ of $\cC$ (of operations) given for each 
collection of colors $c:I\to [\cO]$ ($I$ is a finite set) and a color $d\in [\cO]$.
\item[3.] A composition map defined for each map of finite sets $f:I\to J$, 
with collections of
colors $c:I\to [\cO],\ d:J\to[\cO],\ e\in[\cO]$. This is a map
\begin{equation}\label{eq:composition}
\cO(d,e)\otimes\bigotimes_{j\in J}\cO(c_j,d(j))\rTo\cO(c,e),
\end{equation}
where $c_j$ denotes the restriction of $c:I\to[\cO]$ to $I_j=f^{-1}(j)$.\item[4.]The unit maps $1_c:\one\rTo\cO(\{c\},c)$ for each $c\in[\cO]$, where $\one$ is the unit object of $\cC$.
\end{itemize}
The composition maps are required to satisfy the following associativity  and unit conditions.
\begin{itemize}
\item[1.]Associativity: for a pair of maps $I\rTo^fJ\rTo^gK$, collections of colors $c:I\to[\cO],\ d:J\to[\cO], e:K\to[\cO]$ and 
$f\in[\cO]$, two compositions
$$\cO(e,f)\otimes\bigotimes_{k\in K}\cO(d_k,e(k))\otimes
\bigotimes_{j\in J}\cO(c_j,d(j))\rTo\cO(c,f),$$
with $c_j$ being the restriction of $c$  to $I_j=f^{-1}(j)$ and
$d_k$ being the restriction of $d$ to $J_k=g^{-1}(k)$, coincide.
\item[2.]Left unit: For any $c:I\to[\cO]$, $J=\{j\}$ consisting of
one element and $d(j)=e\in[\cO]$, the map (\ref{eq:composition})
induces the identity map
$$ \cO(c,e)=\one\otimes\cO(c,e)\rTo^{1_e\otimes\id}\cO(d,e)\otimes
\cO(c,e)\rTo\cO(c,e).$$
\item[3.]Right unit: For $f:I\to J$ bijection with $c=d\circ f$, 
the map (\ref{eq:composition}) induces the identity map
$$\cO(d,e)\rTo^{\id\otimes\bigotimes 1_{d(j)}}\cO(d,e)\otimes
\bigotimes_{j\in J}\cO(c_j,d(j))\rTo\cO(c,e).$$
\end{itemize}

Note that colored operads in our definition are ``symmetric'': the right unit axiom defines a canonical action of the automorphism group 
of $c:I\to[\cO]$ on all $\cO(c,d)$.

Colored operads are also known under other names as ``multicategories" 
or ``pseudo-tensor categories".
The first name is self-explanatory. The reason for the second one 
(due to Beilinson) is that colored operads can be assigned to symmetric 
monoidal categories. Thus, colored operads can be seen as the generalizations of symmetric monoidal  categories. Here are the
details.

Let $\cD$ be a symmetric monoidal category enriched over $\cC$ 
(The case $\cC=\Set$ is as interesting as any other). 
We can put $[\cD]=\Ob(\cD)$ and define
$\cD(c,d)=\Hom_\cD(\otimes_{i\in I}c(i),d)$. This yields a colored operad in $\cC$.

It is not difficult to understand when a colored operad $\cO$ 
comes in the above described manner from a symmetric monoidal category. First of all, any colored operad $\cO$ has an underlying category
(also enriched over $\cC$) denoted $\cO_1$ in the sequel. If we wish $\cO$ to come from a symmetric monoidal (SM) category, the functors $d\mapsto\cO(c,d)$ should be 
representable for each $c:I\to[\cO]$. This condition is not yet sufficient: assuming all functors above are representable by the objects denoted as $\otimes_{i\in I}c(i)$, we obtain for each map $f:I\to J$ of finite sets a canonical map
\begin{equation}
\otimes_{i\in I}c(i)\rTo \otimes_{j\in J}(\otimes_{i\in I_j}c(i)).
\end{equation} 
If these maps are isomorphisms, our colored operad $\cO$ comes from a 
SM category (uniquely defined up to equivalence).
\footnote{By the way, defining a SM category as a colored operad
satisfying the above properties, allows one not to care about associativity or commutativity constraints.}

\subsubsection{Planar versions}

If one replaces finite sets with totally ordered finite sets and the maps of
finite sets with the  monotone maps, we get a multicolor notion of planar
(or asymmetric) operad. This notion generalizes the notion of monoidal category in the same way
as the notion of colored operad generalizes the notion of SM category.  

\subsubsection{Maps of operads}
\label{sss:map-of-op}

Given two colored operads $\cP$ and $\cQ$, a map of operads 
$f:\cP\to\cQ$ is defined
as a map $f:[\cP]\to [\cQ]$ together with a compatible
collection of maps
$$ \cP(c,d)\to \cP(f\circ c,f(d))$$
for each $c:I\to[M]$.

Compatibility means that the above maps preserve units and are compatible with compositions.

If $\cP$ and $\cQ$ are the operads corresponding to SM categories, a map of operads $f:\cP\to\cQ$ is  what is usually called {\em a lax  SM functor}.

\subsubsection{SM functors}

In more detail, let $\cP$ and $\cQ$ be SM categories and let $f:\cP\to\cQ$ be a
map of the corresponding colored operads. This means that a compatible collection
$$ \Hom(\otimes_{i\in I}a_i,b)\to\Hom(\otimes_{i\in I}f(a_i),f(b))$$
is given. By naturality, this is the same as a compatible collection
of morphisms
\begin{equation}\label{PT-functor}
\otimes_{i\in I}f(a_i)\rTo f(\otimes_{i\in I}a_i).
\end{equation}

This is what is usually called a lax SM functor.
A  map of operads $f:\cP\rTo\cQ$ is called a SM functor 
if the maps~(\ref{PT-functor}) are isomorphisms.

\subsubsection{Algebras}

We assume that the base symmetric monoidal category $\cC$ admits colimits and
the tensor product commutes with colimits along  each one  of the arguments.

Let $\cO$ be a colored\footnote{In what follows we will say ``operads'' and ``colorless operads'' instead of ``colored operads'' and ``operads''.}
  operad in $\cC$ and let $\cD$ be a SM category enriched over $\cC$. An $\cO$-algebra in $\cD$
is just a map of operads $A:\cO\to\cD$.

The category of $\cO$-algebras in $\cD$ is denoted as $\Alg_\cO(\cD)$ or just
$\Alg_\cO$ if  $\cD$ can be understood from the context.

The following theorem is very standard, see, for instance, \cite{BM}, 1.2.

\begin{thm}\label{thm:adjalg}
Let $f:\cP\to\cQ$ be a map of (small) operads and let $\cD$ be a SM $\cC$-enriched category having colimits. There is a pair of adjoint
functors
$$ f_!:\Alg_\cP(\cD)\pile{\rTo \\ \lTo}\Alg_\cQ(\cD):f^*$$
where $f^*$ is the forgetful functor, assigning to
$A:\cQ\to\cD$ the composition $f^*(X)=A\circ f:\cP\to\cQ\to\cD$.
\end{thm}

In the special case where $\cP=[\cQ]$ is the operad with the same colors as $\cQ$ and with no 
nontrivial operations, the functor $f_!$ is the free algebra functor which is worth of a more detailed description.

Let $V:[\cO]\to\cD$ be a collection of objects of $\cD$  numbered by 
the colors.

The free algebra $\F_\cO(V)$ is the collection of objects 
$\F_\cO(V)_d$, $d\in [\cO]$,
described as follows. Collections $c:I\to[\cO]$ form a groupoid
denoted $\Fin/[\cO]$. To each $c\in\Fin/[\cO]$ we assign the object
\begin{equation}\label{eq:freealg}
\cO(c,d)\otimes\bigotimes_{i\in I} V_{c(i)}.
\end{equation}
This gives rise to a functor
$\Free(V)_d:\Fin/[\cO]\to\cD$; its colimit is the component $\F_\cO(V)_d$
of the free $\cO$-algebra generated by $V$. Note that $\Fin/[\cO]$ is a groupoid.
For $c:I\to[\cO]\in\Fin/[\cO]$ denote $\Sigma_c$ its automorphism group (this is a
 subgroup of the symmetric group $\Sigma_I$). Thus, the free $\cO$-algebra generated by $V$ is given by the formula
\begin{equation}\label{eq:freealg2}
\F_\cO(V)_d=\bigoplus_{c\in\pi_0(\Fin/[\cO])}\cO(c,d)\otimes_{\Sigma_c}
\bigotimes_{i\in I}V_{c(i)},
\end{equation}
where the  direct sum is over a set of representatives of isomorphism classes of
objects in $\Fin/[\cO]$.

The functor $f_!$ carries
the free $\cP$-algebra generated by a collection $V=\{V_c\}_{c\in\cP}$ to the free 
$\cQ$-algebra generated by the collection $d\mapsto \coprod_{c\in[\cP]:f(c)=d}V_d$.

\subsection{DG version}
\label{ss:dg}

From now on we fix a  commutative ring $k$ and 
we study operads and algebras with values in the category $C(k)$ 
of complexes over $k$.

First of all, the category of complexes $C(k)$ admits a model structure, with
quasiisomorphisms as weak equivalences and degree-wise surjective maps as fibrations, see, for example, \cite{haha}.

%Cofibrant objects in this model structure are the retracts of the complexes %constructed by joining consecutive generators with a prescribed value of their %differential.

For a large class of dg operads a model category structure on $\Alg_\cO(C(k))$ can be defined using the adjoint pair of functors
\begin{equation}\label{eq:model-adjunction}
\F_\cO: C(k)^{[\cO]}\pile{\rTo \\ \lTo}\Alg_\cO(C(k)):\G,
\end{equation}
where $\G$ is the forgetful functor and $\F_\cO$ is the free $\cO$-algebra functor.

A map of $\cO$-algebras $f:A\to B$ is called a weak equivalence (resp., a fibration) if $\G(f)$ is a weak equivalence (resp., a fibration). In other words, $f$ is a weak equivalence if for 
each color $c\in[\cO]$ the map $A_c\to B_c$ is a quasiisomorphism of 
complexes. It is a fibration if all maps
$A_c\to B_c$ are surjective. It is called a cofibration if it satisfies
the left lifting property with respect to all trivial fibrations.

\begin{dfn}
An operad $\cO$ in $C(k)$ is called {\sl admissible} 
if the category of algebras $\Alg_\cO(C(k))$ admits a model category structure determined by weak equivalences and fibrations defined as above. 
\end{dfn}

The model category structure on $\Alg_\cO(C(k))$ is cofibrantly generated as it is transferred from the cofibrantly generated model category structure on (collections of)
complexes. In particular, any cofibration is a retract of a transfinite composition of
maps of the form $A\to A\langle x\rangle$ where $x$ is a free variable of a given color $c$, a given degree $d$, and a specified value of $dx\in (A_c)^{d+1}$.  

The operads with a fixed collection of colors $K$ can be described as
the algebras over an appropriate operad whose colors are the finite collections of the elements of $K$. This allows one to define weak equivalence and fibration of a map
of operads with a fixed collection of colors in a usual way. This allows one to define
as well cofibrations for maps of operads via the left lifting property with respect to trivial fibrations. Note that we are not requiring or claiming here
the existence of model structure for such category of dg operads. \footnote{In this we follow \cite{BM}.} 

One has the following
\begin{prp}
A cofibrant operad is admissible.
\end{prp}
The colorless case is proven in \cite{virtual}. The same reasoning proves the colored case.
\qed

Another class of admissible operads ($\Sigma$-split operads) is described in Subsection \ref{ss:sigmasplit}. It includes, for instance, all planar operads, or all operads over $k\supset\Q$.

Note the following criterion of admissibility.

\begin{thm}\label{thm:modelham}
An operad $\cO$ in $C(k)$ is admissible if and only
if for any 
$\cO$-algebra $A$ and for any collection of contractible cofibrant 
complexes $M=\{M_c\}$, $c\in[\cO]$, the natural map
\begin{equation}\label{eq:ham}
A\rTo A\coprod \F_\cO(M)
\end{equation}
is a weak equivalence.
\end{thm}
The proof for colorless operads is given  in \cite{haha}. The same reasoning proves the colored case. \qed

One immediately sees that it is sufficient to check that 
(\ref{eq:ham}) is a weak equivalence for $M=\{M_c\}$ with $M_c=0$ for $c\ne c_0$, and $M_{c_0}$ contractible cofibrant.

\subsection{Change of operad}
\label{ss:cho}
Recall that a map $f:\cP\rTo\cQ$ of operads gives rise to a pair of 
adjoint functors
\begin{equation}
f_!:\Alg_\cP\rlarrows\Alg_\cQ:f^*
\end{equation}
where $f^*$ forgets a part of the structure. One has the following

\begin{thm}
Assume $\cP$ and $\cQ$ are admissible. Then the pair of adjoint 
functors $(f_!,f^*)$ is a Quillen pair.
\end{thm}
\begin{proof}
The forgetful functor obviously preserves fibrations (surjective maps)
and trivial fibrations (surjective quasiisomorphisms).
\end{proof}

One can expect the pair $(f_!,f^*)$ to be a Quillen equivalence under some favorable conditions.

Recall that for a dg operad $\cO$ we denote by $\cO_1$ the underlying dg category
which remembers only unary operations of $\cO$. Passing to the zeroth cohomology of all
$\Hom$ complexes, we get a category $H^0(\cO_1)$.

\begin{dfn}
\begin{itemize}
\item[1.]A map $f:\cP\to\cQ$ is called {\sl a weak equivalence} if
\begin{itemize}
\item[a.] For each $c:I\to[\cP]$
and $d\in[\cP]$ the morphism $\cP(c,d)\to \cQ(f\circ c,f(d))$ is a 
quasiisomorphism.
\footnote{In case $f$ induces bijection on the colors, our notion of weak equivalence 
coincides with the one mentioned in \ref{ss:dg}.}
\item[b.] The functor $H^0(f_1):H^0(\cP_1)\to H^0(\cQ_1)$ is an equivalence 
of categories.
\end{itemize}
\item[2.]A map $f:\cP\to\cQ$ is {\sl a strong equivalence} if instead of (a) 
the following 
stronger condition is fulfilled.
\begin{itemize}
\item[a$^\prime$.] Let $c:I\to[\cP]$ be a collection of colors and $d\in[\cP]$.
Choose a decomposition $c=c'\circ p$ 
\begin{equation}
I\rTo^p J\rTo^{c'}\protect[\cP]
\end{equation}
with $p$ surjective, and let $G$ be the subgroup of automorphisms of $I$ over $J$.
The map
\begin{equation}
\cP(c,d)\otimes_Gk \rTo  \cQ(f\circ c,f(d))\otimes_Gk
\end{equation}
is a quasiisomorphism for all $c,d$ and $p$.
\end{itemize}
\end{itemize}
\end{dfn}

\begin{dfn}An operad $\cO$ in $C(k)$ is called $\Sigma$-cofibrant if for 
each $c:I\to[\cO]$ and $d\in[\cO]$ with $G$ the group of automorphisms of $c$, the complex
$\cO(c,d)$ is a (projectively) cofibrant complex of $G$-modules. 
\end{dfn}

\begin{rem}
Weak equivalence of operads implies their strong equivalence in 
case they are $\Sigma$-cofibrant. 
\end{rem}

\begin{thm}
A strong equivalence of admissible operads $f:\cP\rTo\cQ$ gives rise 
to a Quillen equivalence $(f_!,f^*)$. 
\end{thm}
The colorless case is proven in \cite{virtual}. The proof of \cite{virtual}
directly generalizes to the colored case. Theorem 4.1 of \cite{BM} proves
that weak equivalence of admissible $\Sigma$-cofibrant operads gives rise to a Quillen equivalence. 
\qed

The above observations lead us to the following definition.

\begin{dfn}\label{dfn:hsound}
An operad $\cO$ is called homotopically sound if it is admissible and $\Sigma$-cofibrant.
\end{dfn}

One can define therefore homotopy $\cO$-algebras as algebras over
an operad which is a homotopically sound replacement of $\cO$.

\subsection{$\Sigma$-split operads}
\label{ss:sigmasplit}

In this subsection we present another class of admissible operads in $C(k)$.
 
%%%%%%%%%% correction, April 17

\subsubsection{}

There is an
obvious forgetful functor $\cO\mapsto\cO^\sharp$
assigning to an operad its planar counterpart. The functor $\sharp$
admits a left adjoint functor which we denote $\cO\mapsto \cO^\Sigma$; 
if $\cO$ is a planar operad, the operad $\cO^\Sigma$ has the same colors; it is 
defined by the formula
\begin{equation}
\cO^\Sigma(c,d)=\bigoplus_{\theta:I\simeq\langle n\rangle}\cO((c,\theta),d),
\end{equation}
where $\langle n\rangle=\{1,\ldots,n\}$ is the standard (totally ordered)
$n$-element set, the direct sum is over all bijections $\theta$ and a pair $(c,\theta)$ describes a colored collection $c$ numbered by the totally ordered set $(I,\theta)$.

Let $\cO$ be an operad. Applying the adjoint pair of functors described above, we get a new operad which we will denote $\cO^\Sigma$. It has the same colors as $\cO$ and its complexes of operations are defined by the formulas 
 
\begin{equation}
\cO^\Sigma(c,d)=\bigoplus_{\theta:I\simeq\langle n\rangle} \cO(c,d),
\end{equation}
in the previously explained notation.
 
The composition is defined as follows. Let $f:I\to J$ be a map of sets and let
$c:I\to[\cO]$ and $d:J\to [\cO]$ be collections.
The map
\begin{equation}\label{eq:S-composition} 
 \cO^\Sigma(d,e)\otimes\bigotimes_j \cO^\Sigma(c_j, d(j))\rTo \cO^\Sigma(c,e)
\end{equation}
is defined as follows. Choice of a total order on $J$ together with
a choice of total orders on each fiber $f^{-1}(j)$ defines a
lexicographical total order on $I$: if two elements of $I$ belong to 
different fibers, we compare the fibers, and if they belong to the same 
fiber, we compare them inside the fiber.
With the described above choice of the orderings, the corresponding
component of the map (\ref{eq:S-composition}) is given by the
composition (\ref{eq:composition}) for $\cO$.

For example, if $\cO$ is the operad for commutative algebras, $\cO^\Sigma$ 
is the operad for associative algebras.

\subsubsection{}
One has a canonical map
$$ \pi:\cO^\Sigma\rTo\cO$$
summing up the components corresponding to different orderings
(this is just the the counit of the adjunction).

One defines $\Sigma$-splitting as a collection of splittings 
$t=t^{c,d}:\cO(c,d)\rTo \cO^\Sigma(c,d)$ 
of the canonical map $\pi$ described above, 
satisfying the properties (SPL), (INV), (COM) which will be specified
later on. We will usually omit the superscript $(c,d)$ from the notation.

A $\Sigma$-splitting $t$ is defined by a collection of its components
$t_\theta:\cO(c,d)\to \cO(c,d)$ numbered by different orderings of $I$.

The first two requirements for $\Sigma$-splitting are
\begin{itemize}
\item[(SPL)] The map $t$ splits $\pi$, that is $\sum_\theta t_\theta=\id$.
\item[(INV)] For any isomorphism $f:c'\to c$ (that is, a bijection
$f:I'\to I$ satisfying $c'=c\circ f$) the induced isomorphism
$f^*:\cO(c,d)\to \cO(c',d)$ commutes with $t$. The latter means
that 
$$f^*\circ t_\theta=t_{\theta f}\circ f^*.$$
\end{itemize}

The last requirement of $\Sigma$-splitness is a weak form
of compatibility of the splitting with the compositions.

Let $c:I\to [\cO],\ d:J\to[\cO],\ a:K\to[\cO], a':K'\to[\cO]$ be finite collections
in $\cO$. Let $f:I\to J$ be a map of 
finite sets and let $\phi:a\to a'$ be an isomorphism of collections
(that is, a bijection $\phi:K\to K'$ such that $a=a'\circ\phi$).

Gluing the above data, one gets collections  $c\sqcup a:I\sqcup K\to[\cO]$ and
$d\sqcup a':J\sqcup K'\to[\cO]$, as well as a map of finite sets
$f\sqcup\phi:I\sqcup K\to J\sqcup K'$.

The requirement (COM) describes a compatibility of the splitting
with the composition in $\cO$ 

\begin{equation}
\cO(d\sqcup a',e)\otimes\bigotimes_{j\in J}\cO(c_j,d(j))\rTo \cO(c\sqcup a,e)
\end{equation}
induced by the morphism $f\sqcup\phi$.

We are now able to formulate the third requirement of $\Sigma$-splittings.

(COM) The following diagram is commutative.

\begin{equation}
\begin{diagram}
\cO(d\sqcup a',e)\otimes\bigotimes_{j\in J}\cO(c_j,d(j))& \rTo & 
\cO(c\sqcup a,e) \\
\dTo^t & & \dTo^t \\
\bigoplus_{\eta: J\sqcup K'\simeq\langle |J|+|K|\rangle}
\cO(d\sqcup a',e)\otimes\bigotimes_{j\in J}\cO(c_j,d(j))& &
\bigoplus_{\theta:I\sqcup K\simeq\langle |I|+|K|\rangle}
 & \cO(c\sqcup a,e) \\
\dTo^q & & \dTo^q \\
\bigoplus_{k\in K'}
\cO(d\sqcup a',e)\otimes\bigotimes_{j\in J}\cO(c_j,d(j))& \rTo & 
\bigoplus_{k\in K}\cO(c\sqcup a,e)
\end{diagram}
\end{equation}
The upper vertical arrows in the diagram are defined by splitting of
$\cO(d\sqcup a',d)$ and of $\cO(c\sqcup a,e)$ respectively.
In order to define the lower vertical arrows we will introduce the
following notation. For each ordering $\eta$ of the set $J\sqcup K'$ 
we denote by $\min_{K'}(\eta)$ the smallest element of the subset
$K'$ of $J\sqcup K'$. In the same manner we define $\min_K(\theta)$.
Now the maps $q$ send each $\eta$-component (resp.,
$\theta$-component) to the corresponding $\min_{K'}(\eta)$-component
(resp., $\min_K(\theta)$-component).

\begin{rem}
There is another (stronger) version of $\Sigma$-splitness where
$q$ is replaced with a projection to the sum over orderings of $K'$ 
(resp., of $K$). It seems more satisfactory aesthetically; in this formulation
the condition (INV) is its special case for $I=J=\emptyset$.

This stronger version was used in the definition given in \cite{haha}
for the colorless case.
\end{rem}

\begin{exm}\label{exm:Q}
In the case $k\supset\Q$  the map
$$ t_\theta(m)=\frac{1}{n!}m$$
defines a $\Sigma$-splitting.
\end{exm}

\begin{exm}\label{exm:plsigma}
Let $\cP$ be a planar colored operad and let $\cO=\cP^\Sigma$. 
The canonical map of asymmetric operads $\cP\to \cO^\sharp$ defines
a map of operads $t:\cO\to \cO^\Sigma$ splitting the canonical map $\cO^\Sigma\to \cO$.
This map satisfies obviously the conditions (SPL), (INV), (COM).
\end{exm}

\subsection{Admissibility of $\Sigma$-split operads}
\label{ss:ha}
One has
\begin{thm}
\quad
\begin{itemize}
\item $\Sigma$-split operads in $C(k)$ are admissible.
\item If the components $\cO(c,d)$ of a $\Sigma$-split operad $\cO$
are cofibrant complexes, $\cO$ is homotopically sound.
\end{itemize}
\end{thm}

The second claim of the theorem immediately follows from the first one,
as $\cO(c,d)$ is a direct summand of $\cO^\Sigma(c,d)$ which is cofibrant $\Sigma_c$-complex, where $\Sigma_c$ denotes the group of automorphisms of $c:I\to[\cO]$.

The forgetful functor commutes with filtered colimits. Thus, it is 
sufficient to check that the map $A\to A\coprod \F(H_a)$ is a quasiisomorphism for
$H_a$ standard contractible complex $\Cone(\id_k)[d]$ concentrated at a color $a\in[\cO]$.
The proof of the theorem is given in \ref{sss:proof0}---\ref{sss:proof2} below.

\subsubsection{Extending homotopy to a free algebra}
\label{sss:proof0}
Let $V=\{V_d|d\in [\cO]\}$ be a collection of complexes, 
$\alpha:V\to V$ an endomorphism and $h$ a homotopy of $\alpha$ with $\id_V$, 
that is a degree $-1$ map satisfying the condition
$$ dh=\id_V-\alpha.$$

The endomorphism $\alpha$ induces an endomorphism 
$\F_\cO(\alpha):\F_\cO(V)\to \F_\cO(V)$; we will
present an explicit homotopy between $\id_{\F_\cO(V)}$ and $\F_\cO(\alpha)$ which we will
denote $\F_\cO(h)$. The homotopy $\F_\cO(h)$ will be based on a $\Sigma$-splitting of $\cO$. 
 
Recall that one has a morphism of operads $\pi:\cO^\Sigma\to \cO$ identical
on the colors, as well as a $\Sigma$-splitting $t:\cO(c,d)\to \cO^\Sigma(c,d)$.

We are now ready to define a homotopy $H$ on $\F_\cO(V)$.
Recall that $\F_\cO(V)_d$ is the direct limit of the functor $\Free(V)_d$ 
carrying a collection
$c:I\to [\cO]$
to
$$ \Free(V)_d(c)=\cO(c,d)\otimes\bigotimes_i V_{c(i)}.$$
We will define a degree $-1$ endomorphism $H$ of each separate
$\Free(V)_d(c)$ 
compatible with the isomorphisms $c\to c'$ of collections. It is given
by the composition
\begin{multline}
 \cO(c,d)\otimes\bigotimes_i V_{c(i)}\rTo^t
 \bigoplus_{\theta:I\simeq\langle n\rangle} 
\cO(c,d)\otimes\bigotimes_i V_{c(i)}\rTo^S\\
 \bigoplus_{\theta:I\simeq\langle n\rangle} 
\cO(c,d)\otimes\bigotimes_i V_{c(i)}\rTo^\pi
 \cO(c,d)\otimes\bigotimes_i V_{c(i)},
\end{multline}  
with the map $S$ being defined at the $\theta$-component as
$$S_\theta=\sum_i \id_{\cO(c,d)}\otimes\alpha^{i-1}\otimes h\otimes\id^{n-i}.$$

\subsubsection{} 
\label{sss:proof1}

In order to check that the morphism $A\to A\coprod \F_\cO(H_a)$ is a
quasiisomorphism for  $H_a=\Cone(\id_k)[d]$, one proceeds as
follows.

Let $A'=A\oplus H_a$. Then $A\coprod \F_\cO(H_a)$ can be described 
as the quotient of $\F_\cO(A')$ by the ideal generated by the kernel 
of the natural map $\F_\cO(A)\rTo A$.

Let $\alpha:A'\to A'$ be zero on $H_a$ and $\id_A$ on $A$.
Let $h:A'\to A'$ be the degree $-1$ map vanishing on $A$ such that
$dh=\id-\alpha$. Then $h$ defines a homotopy $\F_\cO(h)$ on 
$\F_\cO(A')$ extending $h$.

Let $\cI$ the the kernel of the natural projection $\F_\cO(A)\to A$ 
and let $\cJ$ be the ideal in $\F_\cO(A')$ generated by $\cI$. 
We check below that 
$H(\cJ)\subset \cJ$ and this induces a homotopy on the quotient 
$\F_\cO(A')/J=A\coprod \F_\cO(H_a)$.

\subsubsection{Action of $\F_\cO(h)$ on $\F_\cO(A')$}
\label{h-on-faprime}

Some relevant notation. For $c:I\to[\cO]$ and $n\geq 0$ we define
$c^{*n}:I\sqcup\langle n\rangle\to[\cO]$ by the formula
$$ c^{*n}(i)=c(i)\text{ for } i\in I; \
c^{*n}(k)=a\text{ for } k\in\langle n\rangle ; 
$$
The $e$-component of the free algebra $\F_\cO(A')$ with 
$A'=A\oplus H_a$ is the colimit
of the complexes
$$ \cO(c^{*n},e)\otimes \bigotimes_iA_{c(i)}\otimes H_a^{\otimes n}$$

The homotopy $\F_\cO(h)$ is defined by the components $S_\theta$
numbered by the total ordering $\theta$ of the set $I\sqcup\langle n\rangle$.
Since $h$ vanishes on $A$ and $\alpha$ is identity on $A$ and vanishes on 
$H_a$, the map 
$S_\theta$ has form
$$S_\theta=\id_{\cO(c^{*n},e)}\otimes\id_A\otimes\id^{\otimes k-1}\otimes h\otimes\id^{n-k}$$
where the homotopy $h$ is applied to the $k$-th component of $H_a$, with
$k:=\min_{\langle n\rangle}(\theta)$.

\subsubsection{End of the proof}
\label{sss:proof2}

We keep the notation of \ref{sss:proof1}.

The ideal 
$\cJ$ in $\F_\cO(A')$
generated by $\cI$, is spanned by the expressions
\begin{equation}\label{eq:j}
u\otimes\delta\otimes\bigotimes_{i\in I-\{0\}}b_i\otimes
\bigotimes_{k\in\langle n\rangle}x_k
\end{equation}
where $c:I\to[\cO],\ 0\in I,\ c(0)=c_0,\ \delta\in\cI_{c_0}$, 
$u\in \cO(c^{*n},e)$, $b_i\in A_{c(i)}$
and $x_k\in H_a$.

We will now explicitly calculate the image of (\ref{eq:j}) under 
the homotopy $\F_\cO(h)=\sum_\theta S_\theta\circ t_\theta$ to 
make sure it belongs to $\cJ$. 

Let 
$$t(u)=\sum_{\theta:I^{*n}\to\langle |I|+n\rangle}t_\theta(u).$$
We claim that $\F_\cO(h)$ carries   (\ref{eq:j}) to the sum
\begin{equation}\label{eq:fj}
\sum_\theta u_\theta\otimes\delta\otimes\bigotimes_{i\in I-\{0\}}b_i\otimes
\bigotimes_{k\in\langle n\rangle}x_{\theta,k},
\end{equation}
where 
\begin{equation}\label{x-theta-k}
x_{\theta,k}=\left\{
\begin{array}{ll} x_k, & k\ne\min_{\langle n\rangle}(\theta)\\
h(x_k), & k=\min_{\langle n\rangle}(\theta)
\end{array}
\right.
\end{equation}

It is sufficient to check the formula (\ref{eq:fj}) in case $\delta$ 
is a monomial in $\F_\cO(A)$:
\begin{equation}\label{eq:delta}
 \delta=m\otimes\bigotimes_{j\in J } a_j
\end{equation}
with $m\in \cO(d,c_0),\ d:J\to[\cO],\ a_j\in A_{d(j)}$.

Replace 
$\delta$ in (\ref{eq:j}) with the expression (\ref{eq:delta}).
We get a monomial
\begin{equation}\label{eq:j2}
z:=u\circ m\otimes \bigotimes_{j\in J } a_j
\otimes\bigotimes_{i\in I-\{0\}}b_i\otimes
\bigotimes_{k\in\langle n\rangle}x_k,
\end{equation} 
where $u\circ m$ denotes the composition of $u$ and $m$ belonging
to $\cO(c\circ d,e)$ where $c\circ d:I-\{0\}\sqcup J\to[\cO]$ is the restriction
of $c\sqcup d$, whose image under $\F_\cO(h)$ is given by the formula
\begin{equation}\label{eq:fj2}
\F_\cO(h)(z)=\sum_{\eta:I-\{0\}\sqcup J\simeq\langle |I|+|J|-1\rangle}
S_\eta\circ t_\eta.
\end{equation}

By the axiom (COM) of $\Sigma$-splitness applied to
the surjection $I-\{0\}\sqcup J\rTo I$ sending the elements of $J$ to $0$
and the elements of $I-\{0\}$ to themselves, we deduce that 
$\F_\cO(h)(z)$ 
is equal to (\ref{eq:fj}).

\subsection{Simplicial structure in characteristic zero}

All operads are $\Sigma$-split when $k\supset\Q$, so in this case the category
of algebras $\Alg_\cO(C(k))$ has a model structure described in Theorem~\ref{thm:modelham}.

Moreover, polynomial differential forms allow one to define a simplicial structure on the category
$\Alg_\cO(C(k))$ which is (partly) compatible with the model category structure.
We will present the definitions and formulate the theorem. The proof is identical to the
colorless case described in \cite{haha}, 4.8.

For $k\supset\Q$ and $n\geq 0$ one defines a dg commutative algebra $\Omega_n$ by the formula
$$\Omega_n=k[x_0,\ldots,x_n,dx_0,\ldots,dx_n]/(\sum x_i-1,\sum dx_i).$$
The assignment $n\mapsto\Omega_n$ defines a simplicial object in the category of commutative dg algebras over $k$. It is canonically extended to a contravariant functor
\begin{equation*}
\Omega:\sSet\rTo\Alg_\Com(C(k))
\end{equation*}
carrying colimits to limits.

For $A,B\in\Alg_\cO(C(k))$ the simplicial set $\Map(A,B)$ is defined by the formula
$$ \Map(A,B)_n=\Hom(A,\Omega_n\otimes B).$$

The compatibility of the simplicial structure on $\Alg_\cO(C(k))$ with the model category structure
is described in the following theorem.

\begin{thm}
Assume $k\supset\Q$ and
let $\cO$ be an operad in $C(k)$.
The category
$\Alg_\cO(C(k))$ of $\cO$-algebras with values in $C(k)$ has a structure
of model category with quasiisomorphisms as weak equivalences and 
componentwise surjective maps as fibrations. 
The category $\Alg_\cO(C(k))$ has a \lq\lq{}weak simplicial model category structure\rq\rq{} (see~\cite{H.L}, 1.4.2), that is a simplicial structure such that the
axioms (M7) and the half of the axiom (M6), see \cite{hirsch}, 9.1.6, are satisfied.
\begin{itemize}
\item[$(\frac{1}{2}M6)$] For every finite simplicial set $K$ and $A\in\Alg_\cO(C(k))$
the \lq\lq{}weak path object\rq\rq{} (see~\cite{H.L}, 1.4.1) $A^K$ exists and is defined by the formula 
$$ A^K=\Omega(K)\otimes A. 
\footnote{It has an $\cO$-algebra structure as the tensor product of a commutative algebra with an $\cO$-algebra.}
$$
\item[$(M7)$]For a cofibration $i:A\to B$ and a fibration $p:X\to Y$ in $\Alg_\cO(C(k))$ the map
of simplicial sets
\begin{equation}\label{}
\Map(B,X)\rTo\Map(A,X)\times_{\Map(A,Y)}\Map(B,Y)
\end{equation}
is a fibration which is trivial if either $i$ or $p$ is trivial.
\end{itemize}
\end{thm}

\subsection{A pushout of algebras}
Let $\cO$ be an admissible operad.
In this subsection we will present, for a later use, explicit formulas for a
specific type of cofibrations of operad algebras.
This is a generalization of formula presented in \cite{SS00} in the proof of Lemma 6.2,
see also F.~Muro's very detailed account of the planar case~\cite{Mu}, Lemmas 8.2, 8.5.

Recall that the model category structure on $\Alg_\cO(C(k))$ is transported from the
projective model structure on $C(k)$ via adjunction~(\ref{eq:model-adjunction}).

Given a cofibration $f:V\to W$ of cofibrant objects in $C(k)^{[\cO]}$, that is a collection
of cofibrations of cofibrant compexes $f_c:V_c\to W_c$ numbered by $c\in[\cO]$,
and a pushout diagram in $\Alg_\cO(C(k))$
\begin{equation}\label{eq:AtoB}
\begin{diagram}
\F_\cO(V) & \rTo^{\F_\cO(f)}& \F_\cO(W) \\
\dTo & & \dTo\\
A& \rTo & B
\end{diagram}
\end{equation} 
with cofibrant $A$, we will explicitly describe the morphism $\G(A)\rTo\G(B)$
as a colimit of a sequence of cofibrations 
\begin{equation}\label{eq:Bi}
\G(A)=B_0\rTo B_1\rTo\ldots\rTo B_k\rTo\ldots
\end{equation}
in $C(k)^{[\cO]}$.

\subsubsection{Enveloping operad}
In order to describe the precise formula for $B_i$, we need a colored version of enveloping operads, as defined and used by Berger and Moerdijk \cite{BM3}. 
Let $\cO$ be a colored operad and let $A$ be a $\cO$-algebra.
The enveloping operad $\cO_A$ can be defined as the operad governing $\cO$-algebras 
$X$ endowed with a morphism of $\cO$-algebras $A\rTo X$. Admissibility criterion
 \ref{thm:modelham} immediately implies that, if $\cO$ is
admissible, then for any $\cO$-algebra $A$ the operad $\cO_A$ is as well admissible.

Furthermore, if $\cO$ is $\Sigma$-cofibrant and $A$ is a cofibrant, the enveloping operad $\cO_A$ is also $\Sigma$-cofibrant. The latter is proven in \cite{BM0}, 5.4,
for colorless operads, but the colored version can be proven in the same way. 

\subsubsection{}
In the formula (\ref{eq:BB}) below we will use the following notation.
Given a collection of maps $\phi_i:X_i\to Y_i,\ i\in I$ in $C(k)$, we can form a functor from the standard $|I|$-cube, considered as a poset of subsets $S\subset I$,
to $C(k)$, carrying $S$ to $Z(S):=\bigotimes_{i\in I} Z_i$ where $Z_i=X_i$ for 
$i\not\in S$ and $Y_i$ for $i\in S$. This yields a map
\begin{equation}\label{eq:wedge}
\bigwedge_{i\in I}\phi_i :\colim_{S\ne I}Z(S)\rTo Z(I)=\otimes_{i\in I}
 Y_i.
\end{equation}
We will need some more notation. Recall that for $c:I\to[\cO]$ we denote $\Sigma_c$ the automorphism
group of $c$.  

Now, fix $d\in[\cO]$. The $d$-component of the map $B_{k-1}\rTo B_k$ is defined as the
 pushout of the map
\begin{equation}
\label{eq:BB}
\bigoplus_{\stackrel{c\in\pi_0(\Fin/[\cO])}{\scriptscriptstyle|I|=k}}
\cO_A(c,d)\otimes_{\Sigma_c}\bigwedge_{i\in I}f_{c(i)},
\end{equation}
where $f_{c(i)}:V_{c(i)}\to W_{c(i)}$ is a component of the map $f:V\to W$.

We have to specify the map from the source of (\ref{eq:BB}) to the $d$-component of $B_{k-1}$. Fixing $c:I\to[\cO]$ with $|I|=k$,   the formula (\ref{eq:BB})
is obtained from a cubic diagram whose vertices are numbered by subsets  $T\subset I$
so that $T$-th vertex has $|T|$ times the factor $W$ and $k-|T|$ times the factor
$V$. We have to specify a map from all such vertices corresponding to $T\ne I$, to $B_{k-1}$. This is done by replacing factors of $V_{c(i)}$ with respective factors of $A_{c(i)}=\cO_A(\emptyset,c(i))$, composing the components of $\cO_A$ via
$$
\cO_A(c,d)\otimes\bigotimes_{i\in I-T}\cO_A(\emptyset,c(i))\rTo \cO_A(c|_T,d),
$$
and, finally, composing the result with the map
$$
\cO_A(c|_T)\otimes\bigotimes_{i\in T}W_{c(i)}\rTo B_{k-1}.
$$

We now have
\begin{prp}\label{prp:pullback}
Let $\cO$-algebra $B$ be given by the pushout diagram~(\ref{eq:AtoB}). Then the 
object $\G(B)\in C(k)^{|\cO|}$ is a colimit of the sequence of cofibrations
(\ref{eq:Bi}) so that $B_{k-1}\to B_k$ is defined by the pushout diagram~(\ref{eq:BB})
in $C(k)^{|\cO|}$.
\end{prp}
\begin{proof}
The $\cO$-algebra $B$ can be presented as a split coequalizer of
$$
A\sqcup \F_\cO(V\oplus W)\stackrel{\rTo^\alpha}{\rTo_\beta}A\sqcup \F_\cO(W), 
$$
where $\alpha$ is determined by the map $(f,\id_W):V\oplus W\to W$, and $\beta$ is $\id_W$ on $W$ and is defined by $V\to A$ on $V$. 
The above coequalizer can be more conveniently rewritten as the coequalizer of
$$
\F_{\cO_A}(V\oplus W)\stackrel{\rTo^\alpha}{\rTo_\beta}\F_{\cO_A}(W). 
$$
We will apply the functor $\G$ before calculating the colimit. The free algebras
considered as collection of complexes, are direct sums of components $\F^k$ corresponding to collections of colors $c:I\to[\cO]$ with $|I|=k$. We define $B_k$ as the image of the map $\F^{\leq k}_{\cO_A}(W)\rTo B$. Then it is obvious that the image of $\F^k_{\cO_A}(W)$ and $B_{k-1}$ generate the whole $B_k$ and that the $d$-component of the fiber product $\F^k(W)\times_{B_k}B_{k-1}$ coincides with the source of the map~(\ref{eq:BB}).
\end{proof}

\begin{crl}\label{crl:excellent}
Assume $\cO$ is homotopically sound.
Let $A$ be a cofibrant $\cO$-algebra. Then for each $d\in[\cO]$ the complex
$\G(A)_d$ is cofibrant.
\end{crl}
\begin{proof}
Any cofibrant $\cO$-algebra is a retract of a transfinite sequence of cofibrations
as in (\ref{eq:AtoB}). Formulas (\ref{eq:BB}) show that each step is a cofibration 
in $C(k)^{[\cO]}$, so the limit is as well a cofibration. Finally, retract of a cofibration is a cofibration.
\end{proof}

\section{SM $\infty$-categories}
\label{s:SM}
 
In this section we will construct certain SM $\infty$-categories
(of $\infty$-categories, of dg categories and others)
by Dwyer-Kan localization. 
We construct an adjoint pair of functors
$$ \fC_\dg:\Cat_\infty^\times\rlarrows N(\dg\Cat)^\otimes:\fN_\dg$$
between the symmetric monoidal $\infty$-categories of infinity-categories and of dg categories, induced by Dold-Kan equivalence. The functor $\fC_\dg$ is symmetric monoidal, whereas $\fN_\dg$ is lax symmetric monoidal (that is, a morphism of $\infty$-operads).

We will denote $\Cat^\SM_\infty:=\Alg_\Com(\Cat_\infty)$
and  $\dg\Cat^\SM=\Alg_\Com(N(\dg\Cat))$.

The functor $\fN_\dg$ induces  a functor
$$ \fN_\dg:\dg\Cat^\SM \rTo  \Cat^\SM_\infty.$$

Furthermore, we will see that both $\infty$-categories $\dg\Cat^\SM(k)$ and $\Cat^\SM_\infty$
 are enriched over $\Cat_\infty$, so that $\fN_\dg$ preserves this enrichment. The latter
 means that for a pair of symmetric monoidal dg categories $C,\ D$ one has a functor
$$ \Fun^\otimes(C,D)\rTo\Fun^\otimes(\fN_\dg(C),\fN_\dg(D)).$$ 

\subsection{Localization}

Given a symmetric monoidal $\infty$-category $\cC^\otimes$ and a collection of arrows $W$, we would like to be able to define a SM structure on the localization $\cL(\cC,W)$. This is easy if
the tensor product preserves $W$, see \cite{H.L}, 3.2 or \cite{L.HA}, 4.1.3.4.

In this case the localization of the 
total category $\cC^\otimes$ with respect to the collection of arrows in $\cC^\otimes$ generated by $W$, yields what we call a strict SM localization: this is a SM functor
\begin{equation}
\cC^\otimes\rTo\cL(\cC^\otimes,W^\otimes)
\end{equation}
universal among SM functors $\cC^\otimes\to\cD^\otimes$ carrying $W$ to equivalences.
Moreover, the underlying $\infty$-category of the strict SM localization is the localization 
$\cL(\cC,W)$ and the localization functor is also universal among lax monoidal functors
$\cC^\otimes\to\cD^\otimes$, see \cite{H.L}, 3.2.

Strict SM localizations seldom exist: tensor product does not always preserve weak equivalences.
In this paper we will use the following ad hoc construction. Given a SM $\infty$-category $\cC^\otimes$ and a collection of arrows $W$ in $\cC$, we will present a full subcategory
$\cC_0^\otimes$ such that
\begin{itemize}
\item The pair $(\cC^\otimes_0,W_0=W\cap\cC_0)$ admits a strict SM localization.
\item The embedding $\cC_0\rTo\cC$ induces an equivalence of the localizations
$ \cL(\cC_0,W_0)\to\cL(\cC,W)$.
\end{itemize} 

We will call $\cL(\cC_0^\otimes,W_0^\otimes)$ the SM localization of $(\cC^\otimes,W)$. This
construction depends, in general, on the choice of $\cC_0$. We believe that in the examples
below, it satisfies a universal property which makes it {\sl right SM localization}
as defined in \cite{H.L}, 3.3. By~\cite{H.L}, 3.3.3, this is so in Example \ref{sss:QC}.

\subsubsection{Example: $\QC(k)^\otimes$} 
\label{sss:QC}

Let $k$ be a commutative ring, $\cC:=C(k)$ the category of complexes of $k$-modules
and let $W$ be the collection of quasiisomorphisms. We choose $\cC_0$ to be 
the full subcategory of cofibrant complexes. As a result, we get a SM $\infty$-category denoted as $\QC(k)$. This is the SM $\infty$-category of $k$-modules; its homotopy category is the derived category of $k$.
 
\subsubsection{Example: categories of enriched categories} 
\label{sss:CatQC}

We will use a similar construction to define SM $\infty$-categories 
of certain enriched categories. The corresponding model categories
were defined by G.~Tabuada, see~\cite{T1,T2}. These are
\begin{itemize}
\item $\dg\Cat(k)$, the category of categories enriched over $C(k)$.
\item $\dg^{\leq 0}\Cat(k)$, that of categories enriched over 
$C^{\leq 0}(k)$.
\item $\sMod\mhyphen\Cat(k)$, that of categories enriched over the simplicial $k$-modules.
\end{itemize}

Symmetric monoidal structure on all these categories is  induced by the symmetric monoidal structure on $C(k),\ C^{\leq 0}(k)$ and $\sMod(k)$ respectively.  In each one of the cases
$\cC:=\dg\Cat,\ \dg^{\leq 0}\Cat$ or $\sMod\mhyphen\Cat$, the full subcategory
$\cC_0$ is spanned by the categories whose $\Hom$-objects are cofibrant. 

In all three cases tensor product preserves weak equivalence of categories belonging to $\cC_0$.
It remains to check that in all three cases the embedding
$\cC_0\rTo\cC$ induces an equivalence of DK localizations.
This is routinely done using Key Lemma 1.3.6 of \cite{H.L}\footnote{see 1.3.7 of
\cite{H.L} for the routine.}.

This yields symmetric monoidal $\infty$-categories which we denote 
$N(\dg\Cat)^\otimes$, $N(\dg^{\leq 0}\Cat)^\otimes$ and 
$N(\sMod\mhyphen\Cat)^\otimes$.

\subsection{Dold-Kan correspondence}

\subsubsection{Classical Dold-Kan equivalence}

Here we will fix some notation. The functor of normalized chains
$$ C_*:\sMod(k)\rTo C^{\leq 0}(k)$$
from simplicial $k$-modules to nonpositively graded complexes
of $k$-modules
is well-known to be an equivalence, with the inverse functor
$$ N_*: C^{\leq 0}(k)\rTo \sMod(k)$$
defined by the formula $N_n(X)=\Hom(C_*(\Delta^n),X)$.

The functor $C_*$ is not symmetric monoidal, but it is very close to be one. One has functorial maps

\begin{equation}\label{eq:EM}
\EM_{X,Y}:C_*(X)\otimes C_*(Y)\rTo C_*(X\otimes Y)
\end{equation}
(Eilenberg-MacLane, or shuffle, map),
and
\begin{equation}\label{eq:AW}
\AW_{X,Y}:C_*(X\otimes Y)\rTo C_*(X)\otimes C_*(Y)
\end{equation}
(Alexander-Whitney map) such that
\begin{itemize}
\item The functor $C_*$ is lax symmetric monoidal via $\EM$.
\item It is also colax monoidal 
\footnote{but not colax symmetric monoidal!} via $\AW$
(equivalently, $N_*$ is lax monoidal via $\AW$).
\item Both $\AW$ and $\EM$ are homotopy equivalences and 
$\AW\circ\EM=\id$.
\end{itemize}

\subsubsection{Enriched categories}
Any lax monoidal functor $F:M\rTo N$ induces a functor
$$F:\Cat_M\rTo\Cat_N$$
between the respective enriched categories.

Therefore, one has a pair of functors
\begin{equation}\label{eq:barCN}
\bar C:\sMod\mhyphen\Cat(k)\rlarrows\dg^{\leq 0}\Cat(k):\ \bar N,
\end{equation}
where $\bar C=(C_*,\EM)$ and $\bar N=(N_*,\AW)$,
with a natural isomorphism $\bar N\circ\bar C=\id$.
Note that the  functors $\bar C,\bar N$ do not form an adjoint pair.

The functor $\bar C$ is lax symmetric monoidal. Developing the ideas
of \cite{SS}, Tabuada proved in \cite{T2} that the functor $\bar C$
has a left adjoint and this pair defines a Quillen equivalence.

Therefore, an equivalence 
\begin{equation}
N(\bar C):N(\sMod\mhyphen\Cat(k))\rTo N(\dg^{\leq 0}\Cat(k))
\end{equation}
is induced. Since it is lax SM, it is a symmetric monoidal equivalence.
Since $N(\bar N)$ is left inverse, it is an inverse symmetric monoidal
equivalence.
 
\subsubsection{}

We can now define an adjoint pair 
\begin{equation}
\fC_\dg:\Cat^\times_\infty\rlarrows N(\dg\Cat(k))^\otimes\ :\fN_\dg
\end{equation}
as a composition
\begin{equation}
N(\sCat^\times)\rlarrows N(\sMod\mhyphen\Cat(k))^\otimes \rlarrows
N(\dg^{\leq 0}\Cat(k))^\otimes \rlarrows N(\dg\Cat(k))^\otimes.
\end{equation}
The adjoint pair in the middle is a SM equivalence. In two other adjoint pairs the left adjoint is symmetric monoidal, therefore the right adjoint is a map of operads, see Appendix~\ref{sec:SMA}.

Thus, $\fC_\dg$ is symmetric monoidal and $\fN_\dg$ is a map of operads. 

\subsubsection{}

We will now show that the functor $\fN_\dg$  carries
the dg category $C(k)$ to $\QC(k)\in\Cat_\infty$. Moreover,
$\fN_\dg$ carries the commutative algebra object $C(k)^\otimes$ to $\QC(k)^\otimes$.

Let $C^c_\sharp(k)$ denote the category of cofibrant complexes of $k$-modules.
It is $k$-linear and so is enriched in a \lq\lq{}trivial\rq\rq{} way over $C(k)$:
\begin{equation}
\chom_\sharp(X,Y):=Z^0(\chom(X,Y)).
\end{equation}

The functor $\fN_\dg$ carries $C^c_\sharp(k)$ into 
the category whose Hom object are discrete (as simplicial objects)
$k$-modules. We identify $\fN_\dg(C^c_\sharp(k))$ with $C^c_\sharp(k)$
for the obvious reason.

The functor $\fN_\dg$ applied to the map $C^c_\sharp(k)\to C^c(k)$
yields a map 
$$C^c_\sharp(k)\rTo \fN_\dg(C^c(k))$$ 
which carries quasiisomorphisms to equivalences. Therefore, a map
$$\QC(k)\rTo \fN_\dg(C^c(k))$$ 
is induced. It is an equivalence by  \cite{H.L}, 1.4.3.

Let us show that $\fN_\dg$ also preserves the symmetric monoidal structure of $C^c(k)$.

The adjoint pair

$$ \fC_\dg:\Cat_\infty^\times \rlarrows N(\dg\Cat(k))^\otimes:\fN_\dg$$
gives rise to an adjoint pair of functors between the 
$\infty$-categories of commutative algebras in respective categories,
\begin{equation}
\label{eq:Cdg-sm}
\fC_\dg:\Cat^\SM_\infty \rlarrows \dg\Cat^\SM(k):\fN_\dg,
\end{equation}
that is between symmetric monoidal $\infty$-categories and (weak) symmetric monoidal dg categories.

We claim that $\fN_\dg$ carries the symmetric monoidal dg category $C^c(k)^\otimes$ to
$\QC(k)^\otimes$ as constructed in \ref{sss:QC}. 

The dg category $C^c_\sharp(k)$ has a symmetric monoidal structure
and the map $C^c_\sharp(k)\rTo C^c(k)$ is a symmetric monoidal functor.
Therefore, the induced arrow
$$ C^c_\sharp(k)^\otimes\rTo \fN_\dg(C^c(k)^\otimes)$$
is also a symmetric monoidal functor. By universality of symmetric monoidal localization we get a symmetric monoidal functor
\begin{equation}
\QC(k)^\otimes\rTo\fN_\dg(C^c(k)^\otimes).
\end{equation}
Since we already know that the induced functor $\QC(k)\to\fN_\dg(C^c(k))$ is an equivalence, it is an equivalence 
of symmetric monoidal $\infty$-categories.

\subsection{$\fN_\dg$, enriched}

We will now show that the $\infty$-categories $\Cat_\infty^\SM$ and $\dg\Cat^\SM(k)$ are enriched over $\Cat_\infty$ and the functor $\fN_\dg$ defined in (\ref{eq:Cdg-sm}) preserves this enrichment.  More precisely, we will present, for a pair $A,B$ of symmetric monoidal
dg categories,  
a map of $\infty$-categories
\begin{equation}  
 \Fun^\otimes(A,B)\rTo\Fun^\otimes(\fN_\dg(A),\fN_\dg(B))
\end{equation}
of respective symmetric monoidal functors extending the map of 
of spaces of morphisms defined by the functor $\fN_\dg$.

We will first explain the construction in the setup of conventional categories, and then
will provide the $\infty$-categorical generalization, using the formalism of SM adjunction
(see Appendix \ref{sec:SMA}).

\subsubsection{A general setup (conventional categories)}
\label{sss:conventional}
Let
\begin{equation}\label{eq:smadj}
\lambda:\cC\rlarrows\cD:\rho
\end{equation}
be an adjoint pair of functors
between symmetric monoidal categories, so that $\lambda$ is symmetric monoidal
(and therefore $\rho$ is lax symmetric monoidal). We assume that $\cD$ is cotensored over $\cC$, which means that there exists a functor $\eta:\cC^\op\times\cD\rTo\cD$, $(X,A)\mapsto A^X$, right
adjoint to the bifunctor $\cC\times\cD\rTo\cD$ carrying the pair $(X,A)$ to 
$\lambda(X)\otimes A$. One can easily see that $\eta$ is lax symmetric monoidal.
 
Assume now that $\cC$ is cartesian. Then any object $X\in\cC$ has an obvious coalgebra structure defined by the diagonal. This implies that for any commutative algebra $A$ in $\cD$ and any object $X$ in $\cC$ the power object $A^X$ has a commutative algebra structure.
The multiplication in $A^X$ is given by the composition
$$ A^X\otimes A^X\rTo (A\otimes A)^{X\times X}\rTo A^{X\times X}\rTo A^X.$$
We can therefore define inner hom 
 on $\Alg_\Com(\cD)$ by the formula
\begin{equation}\label{eq:funtensor}
\Hom(X,\Fun^\otimes(A,B))=\Hom_{\Alg_\Com(\cD)}(A,B^X),
\end{equation}
provided the right-hand side is representable.

The $\cC$-enrichment on $\Alg_\Com(\cD)$ so defined is functorial in the following sense.
Given a sequence of adjoint pairs
$$ \cC\pile{\rTo^{\lambda_1}\\ \lTo_{\rho_1}}\cD_1\pile{\rTo^{\lambda_2}\\ \lTo_{\rho_2}}\cD_2$$
between SM categories, satisfying the above properties, one has a natural isomorphism
\begin{equation}
\rho_2(B^X)=\rho_2(B)^X,
\end{equation}
which induces a canonical map
\begin{equation}
\Fun^\otimes(A,B)\rTo\Fun^\otimes(\rho_2(A),\rho_2(B)).
\end{equation}

\subsubsection{Construction for SM $\infty$-categories}

The only claim requiring a special attention when extending the above construction to $\infty$-categories is the structure of lax symmetric monoidal functor on $\eta:\cC^\op\times\cD\rTo\cD$
induced by the adjunction~(\ref{eq:smadj}). The functor $\lambda$ leads to a $\cC$-left-tensored structure on $\cD$ given by a SM functor
$$ \cC\times\cD\rTo \cD$$
defined by the formula $(a,x)\mapsto \lambda(a)\otimes x.$

The corresponding functor $\cC^\op\times\cD^\op\times\cD\rTo\cS$
carrying the triple $(a,x,y)$ to $\Map(a\otimes x,y)$ is then lax monoidal by~\ref{crl:laxadjunction}. 
Existence of $\cC$-cotensor structure on $\cD$ is equivalent
to $\{1,3\}$-representability of this functor. Once more, according to~\ref{crl:laxadjunction},
this implies that the functor
\begin{equation}\label{eq:eta}
\eta:\cC^\op\times\cD\rTo \cD
\end{equation}
is lax SM, see~\ref{exm:2}.

Now, a lax SM functor induces a functor between respective $\infty$-categories of commutative algebras. A commutative algebra in $\cC^\op\times\cD$ is a pair $(X,A)$ where $X\in\cC$ and $A\in\Alg_\Com(\cD)$. This yields a required functor
 \begin{equation}\label{eq:eta-SM}
\eta^\SM:\cC^\op\times\Alg_\Com(\cD)\rTo\Alg_\Com(\cD),
\end{equation}
which allows one to define $\cC$-valued inner Hom on $\Alg_\Com(\cD)$
by the formula (\ref{eq:funtensor}).

\

\subsubsection{}
We wish to apply the above construction to $\cC=\cD_1=\Cat_\infty$, $\cD_2=N(\dg\Cat(k))$.
The $\infty$-categories $\Cat_\infty$ and $N(\dg\Cat(k))$ can be described as $\infty$-categories
underlying combinatorial model categories, see \cite{L.T}, 2.2.5.1 and \cite{T1},
Thm 1.8.

Therefore, the corresponding underlying $\infty$-categories are presentable.
The tensor products in these $\infty$-categories commute with colimits, so by Corollary 3.2.3.5 of \cite{L.HA} the categories
of commutative algebras in $\Cat_\infty$ and $N(\dg\Cat(k))$ are as well presentable. Furthermore, the functors $\eta$ and $\eta^\SM$ preserve limits in each of the arguments.

This implies presentability of the $\cC$-valued inner Hom given by (\ref{eq:eta-SM}).

\

We now have to make sure that the $\Cat_\infty$-enrichment of $\Cat^\SM_\infty$ defined by the above universal construction, coincides with the standard one, see \cite{L.HA}, 
Definition 2.1.3.7.

\subsubsection{Inner Hom for SM $\infty$-categories}

In case $\cC=\cD$ has products,  the functor $\eta^\SM$ defined in (\ref{eq:eta-SM})
can be described much easier: the functor $\eta$ (\ref{eq:eta}) preserves products in the second variable; therefore, it carries algebras to algebras. Here is an explicit description 
of $\eta^\SM$ for $\cC=\Cat_\infty$, where commutative algebras in $\Cat_\infty$ are 
presented as $\infty$-categories cocartesian over $N\Fin_*$, see \cite{L.HA}, Section 2.

Given a simplicial set $X$ and a SM $\infty$-category $p:B\rTo N\Fin_*$, we define
a simplicial set $B^X$ with a map $q:B^X\rTo N\Fin_*$ as follows. The $n$-simplices of $B^X$
over $\sigma:\Delta^n\rTo N\Fin_*$ are the commutative diagrams
\begin{equation}\label{eq:BX}
\begin{diagram}
\Delta^n\times X & \rTo^{\wt\sigma} & B \\
\dTo^{\mathrm{pr}_1} & & \dTo^p \\
\Delta^n & \rTo^\sigma & N\Fin_*
\end{diagram}.
\end{equation}
A diagram (\ref{eq:BX}) with $n=1$ represents a cocartesian lifting of $\sigma$ iff  
the restriction of $\wt\sigma$ to each vertex of $X$ is cocartesian. This implies that $q$
is a cocartesian fibration; the fiber of $q$ at $\langle n\rangle$ is $B_{\langle n\rangle}^X$,
so $q:B^X\rTo N\Fin_*$ is a SM $\infty$-category. 
 
Now the identity 
\begin{equation}
\Map(X,\Fun^\otimes(A,B))=\Map_{\Cat^\SM_\infty}(A,B^X)
\end{equation}
can be easily verified, which proves that $\Cat_\infty$-valued function space defined
be our general construction is the conventional one for $\cC=\Cat_\infty$.

\section{Rectification of algebras}
\label{sec:rectification}

\subsection{Introduction}
\label{ss:introrec}

Let $\cO$ be a topological operad (that is, a fibrant simplicial operad) with the set of colors $[\cO]$.
We denote $\cO^\otimes$ the corresponding $\infty$-operad in the sense of Lurie
~\cite{L.HA} which is defined as follows. 

Let $\Fin_*$ denote the category of finite pointed sets. 
Its objects are finite pointed
sets $I_*=I\sqcup\{*\}$ and the maps $f:I_*\to J_*$ satisfy
$f(*)=*$. 

We will define first of all a simplicial category $\wt{\cO}^\otimes$ over $\Fin_*$, 
and then will put $\cO^\otimes$ to be the (homotopy coherent) nerve 
of the simplicial category $\wt\cO^\otimes$, see~\cite{L.T}, 1.1.5.5. 
Here is the definition of $\wt\cO^\otimes$.

Its objects over $I_*\in\Fin_*$ are maps
$c:I\to[\cO]$ and the simplicial sets of morphisms over $f:I_*\to J_*$
defined by the formula
$$
\Map^f_{\wt\cO^\otimes}(c,d)=\prod_{j\in J}\cO(c|_{f^{-1}(j)},d(j)).
$$

The composition in $\wt{\cO}^\otimes$ is determined by the composition in $\cO$, see
the details in~\cite{L.HA}, 2.1.1.22.

Fix a commutative ring $k$. We are mostly interested in algebras over $\cO^\otimes$ with values in the SM $\infty$-category $\QC(k)$ of complexes of $k$-modules described in detail
in \ref{sss:QC}.

We want to compare the $\infty$-category $\Alg_\cO(\QC(k))$, as defined in
Lurie's book \cite{L.HA}, 2.1.3.1 (this is just the $\infty$-category
of operad maps $\cO^\otimes\rTo\QC(k)$), with the category 
of ``strict" $\cO$-algebras $\Alg^{\st}_\cO(C(k))$ defined as in 
Section~\ref{sec:models-a}. \footnote{Note that we have changed the notation in order
to distinguish two notions of $\cO$-algebra!}

Assume now we are given a quasiisomorphism of operads
 $\cR\to C_*(\cO,k)$ with $\cR$ homotopically sound. 
 
In this case, as we know,
the category  $\Alg^{\st}_\cR(C(k))$ admits a model structure
with quasiisomorphisms as weak equivalences and  surjective maps as fibrations.

Applying the nerve construction (see \cite{H.L}, 1.3) to $\Alg^\st_\cR(C(k))$, we get
an $\infty$-category. A certain effort is required in order to be able to interpret
a strict $\cR$-algebra as an object of $\Alg_\cO(\QC(k))$. Unexpectedly, the problem exists
even if $\cR=C_*(\cO)$. The reason is that
the singular chains functor $C_*:\sSet\rTo C(k)$ is not symmetric monoidal.

The construction of the functor $\Alg^\st_\cR(C(k))\rTo\Alg_\cO(\QC(k))$ is explained in Subsection~\ref{ss:phi} below. Once we have this functor,
the universal property of the $\infty$-localization 
yields an $\infty$-functor 
\begin{equation}\label{eq:Phi}
\Phi:N(\Alg_\cR^\st(C(k))\rTo \Alg_\cO(\QC(k)).
\end{equation}

Here is the central result of this paper.
\begin{thm}\label{thm:rect-alg}
Let $\cO$ be a topological operad and let $\cR\to C_*(\cO)$
be a homotopically sound replacement. Then the functor $\Phi$ (\ref{eq:Phi}) is an equivalence.
\end{thm}
\subsubsection{}
\label{sss:onefoot}
The proof follows the idea of the proof of \cite{L.HA} 4.1.4.4
dealing with rectification of associative algebras. An $\infty$-categorical 
version of Barr-Beck theory \cite{L.HA}, 6.2, allows one to present an 
$\cO$-algebra $A$ in $\QC(k)$ as a colimit of its
monadic Bar-resolution. The latter consists of free algebras which can be 
easily lifted to $\Alg_\cR^\st(C(k))$.
The homotopy colimit of this simplicial object in $\Alg_\cR^\st(C(k))$
gives the lifting of $A$.
Further details of the proof are given in \ref{thm:proof} below.

But first of all we have to define the map $\Phi$ in a greater detail.

\subsection{Construction of $\Phi$}
\label{ss:phi}

Let $\cO$ be a fibrant simplicial operad and let $\cR\to C_*(\cO,k)$
be a quasiisomorphism of dg operads (bijective on colors) with $\cR$ homotopically sound.
 
The simplicial operad $\cO$ generates a simplicial PROP $P_\cO$ defined as follows. 
Let $\Fin$ be the category of finite sets. It can be considered as a subcategory of $\Fin_*$ via the functor $\Fin\to\Fin_*$ carrying a finite set $I$ to the pointed set $I_*$. One defines the simplicial category $P_\cO$ as the fiber product 
$\wt\cO^\otimes\times_{\Fin_*}\Fin$. The objects of $P_\cO$ are finite
 collections of colors of $\cO$, and the arrows are composed from the operations in $\cO$. The symmetric monoidal structure is defined
by disjoint union of collections.

The image of $P_\cO$  under $\fC_\dg$, see formula (\ref{eq:Cdg-sm}), is presented by the SM dg category $P_{C_*\cO}$ which is the $k$-linear PROP generated by dg operad $C_*(\cO)$.
We will replace $P_{C_*\cO}$ with an equivalent SM dg category $P_\cR$, the PROP generated by the dg operad $\cR$.

Any cofibrant $\cR$-algebra $A$ in $C(k)$ 
gives rise to a symmetric monoidal dg functor

$$ A: P_\cR\rTo C(k)$$

such that all $A(x),\ x\in P_\cR,$ are cofibrant, 
see~ \ref{crl:excellent}. 

This yields, in particular, an arrow in $\dg\Cat^\SM(k)$. Applying to it the functor $\fN_\dg$,
and composing with the unit map, we get
$$ P_\cO\rTo \fN_\dg(P_\cR)\rTo\fN_\dg(C(k))=\QC(k)^\otimes.$$
The above construction  defines a composition
\begin{eqnarray}
\Alg^\st_\cR(C(k))^c\to\Fun^\otimes(P_{\cR},C(k))\simeq\Fun^\otimes(P_{C_*\cO},C(k))\to\\
\nonumber\to\Fun^\otimes(\fN_\dg(P_{C_*\cO}),\QC(k)^\otimes)\to\Fun^\otimes(P_\cO,\QC^\otimes)=\Alg_\cO(\QC(k)).
\end{eqnarray}

Thus, we have a functor
$$ \phi:\Alg^\st_\cR(C(k))^c \rTo \Alg_\cO(\QC(k)).$$
 
It remains to check that $\phi$ carries weak equivalences of cofibrant 
$\cR$-algebras to an equivalence. This is really easy: the functor $\phi$ constructed above
commutes with the forgetful functors $G^\st$~
\footnote{we denoted it $\G$ in Section 2.} and $G$ in the following diagram

\begin{equation}\label{eq:Phi-G}
\begin{diagram}
\Alg^\st_\cR(C(k))^c & \rTo^\phi & \Alg_\cO(\QC(k)) \\
\dTo^{G^\st} & & \dTo^G \\
(C(k)^c)^{[\cO]} & \rTo^{\phi^\triv} & \QC(k)^{[\cO]}
\end{diagram},
\end{equation}
where $\phi^\triv$ is defined by  localization.

Since weak equivalences in $\Alg^\st_\cR$ are detected by $G^\st$ and since the functor $G$ is conservative
(see \cite{L.HA}, Lemma 3.2.2.6), the assertion follows.

\subsection{Proof of \ref{thm:rect-alg}}
\label{thm:proof}

\subsubsection{}
Look at the commutative diagram obtained from (\ref{eq:Phi-G}) by application of the nerve functor to $G^\st$.
\begin{equation}\label{eq:Phi-N}
\begin{diagram}
N(\Alg^\st_\cR(C(k)) && \rTo^\Phi && \Alg_\cO(\QC(k)) \\
&\rdTo^{NG^\st} & & \ldTo^G &\\
& & \QC(k)^{[\cO]} &&
\end{diagram}.
\end{equation}
 
The reasoning briefly explained in \ref{sss:onefoot} is formalized in  
Corollary 6.2.2.14 of \cite{L.HA}. It claims that 
the map $\Phi$ in (\ref{eq:Phi}) is an equivalence, provided the following properties are verified. 

\begin{itemize}
\item[1.] The functor $G$ is conservative.
\item[1$^{\st}$.] The functor $NG^\st$ is conservative.
\item[2.] The functor $G$ admits a left adjoint functor $F$.
\item[2$^{\st}$.] The functor $NG^\st$ admits a left adjoint functor $F^\st$.

\item[3.] Any $G$-split simplicial object in $\Alg_\cO(\QC(k))$ has a colimit and $G$ preserves this colimit.
\item[3$^{\st}$.] The same for $NG^\st$.
\item[4.] The unit map $X\to NG^\st F^\st(X)=G\Phi(F^\st(X))$ induces an 
equivalence
$$ F(X)\rTo \Phi(F^\st(X)).$$
\end{itemize}
Let us check the assertions 1--4.

The  properties 1,2,3 is are proven in \cite{L.HA}, see 3.2.2.6, 3.1.3.5 and 3.2.3.1.
\footnote{Note that 3.1.3.5 is applicable since tensor product in $\QC(k)$ commutes 
with colimits along each one of the arguments.} 

The functor $F^\st$ is obtained by application of the 
the nerve construction (see Proposition~1.5.1, \cite{H.L})
to the functor $\F_\cR$ which is left adjoint to the forgetful functor $G^\st:\Alg^\st_\cR(C(k))\to C(k)^{[\cO]}$. This proves 2$^\st$.

The functor $NG^\st$ is conservative as weak equivalences in $\Alg_\cR^\st$ are detected by $G^\st$.
Thus, it remains to verify the assertions 3$^\st$ and 4.

{\sl Assertion 3$^\st$.}  

According to~\cite{H.L}, 1.5.2, colimits in an $\infty$-category underlying a combinatorial model category can be calculated via derived colimits in the model category. This is applicable to
both $C(k)$ and $\Alg^\st_\cR(C(k))$.

A simplicial complex $X\in C(k)^{\Delta^\op}$ will be called 
{\sl $\colim$-adapted} if 
\begin{itemize}
\item The canonical map $\Left\colim X\rTo\colim X$ is a quasiisomorphism.
\item For all $n\in\Delta^\op$ the components $X_n\in C(k)$
are cofibrant.
\item The complex $\colim X$ is cofibrant.
\end{itemize}

Simplicial objects in $\Alg^\st_\cR(C(k))$ are algebras over a certain operad which we will denote $\cR^{\Delta^\op}$. In Lemma~\ref{lem:simplalg} below we check that
the operad $\cR^{\Delta^\op}$ is also homotopically sound.  

Since the category $\Delta^\op$ is sifted, the colimit over 
$\Delta^\op$ commutes with the forgetful functor $G$.

Therefore, in order to deduce Assertion 3$^\st$, it remains to prove
that the forgetful functor 
$$\Alg_{\cR^{\Delta^\op}}(C(k))=\Alg_\cR(C(k))^{\Delta^\op}\rTo
C(k)^{[\cR]\times\Delta^\op}$$
carries cofibrant simplicial algebras to $[\cR]$-collections of $\colim$-adapted simplicial
complexes. This is also proven in Lemma~\ref{lem:simplalg} below.

Let $C$ be a (small) category and $\cR$ a dg operad. We will now describe a dg operad $\cR^C$ such that
$\cR^C$-algebras are precisely the functors $C\to\Alg_\cR(C(k))$.

The colors of $\cR^C$ are pairs $(c,m)$ where $c$ is a color of $\cR$
and $m\in C$.

A collection of colors $I\to[\cR^C]$ is given by a pair $(c,m)$ where $c:I\to[\cR]$
is a collection of colors in $\cR$ and $m$ a function $I\to\Ob(C)$. 
The complex $\cR^C((c,m),(d,n))$
is defined as $\prod_i\Hom_C(m_i,n)\times\cR(c,d)$, and the composition
in $\cR^C$ is defined by the compositions in $\cR$ and in $C$.

\begin{lem}\label{lem:simplalg}
Assume $\cR$ is a homotopically sound operad.
Then 
\begin{itemize}
\item[a)] $\cR^{\Delta^\op}$ is also homotopically sound.
\item[b)]  the forgetful functor
$$ G:\Alg_{\cR^{\Delta^\op}}(C(k))\rTo C(k)^{[\cR]\times\Delta^\op}$$
carries cofibrant algebras to collections of $\colim$-adapted simplicial complexes.
\end{itemize}
\end{lem}
\begin{proof} 

{\sl Assertion (a).}

We prove that $\cR^C$ is homotopically sound for any category $C$.
First of all, let us prove admissibility of $\cR^C$. 
Let $A$ be an $\cR^C$ algebra and $M$ a cofibrant contractible complex.
Choose a color $(c,m)$ of $\cR^C$. 
We have to prove that the map $A\to B$
is a weak equivalence, where $B$ is obtained from $A$ by freely
joining $M$ at color $(c,m)$. In other words, we have to check that 
for each $n\in C$ the map $A(n)\to B(n)$ is a weak equivalence. Note that
$B(n)$ is freely generated over $A(n)$ as $\cR$-algebra by 
$\Hom_C(m,n)\times M$ which is also cofibrant and contractible. 
Therefore, $\cR^C$ is admissible by the criterion \ref{thm:modelham}.

It remains to make sure $\cR^C$ is $\Sigma$-cofibrant.
The automorphism group of a collection $(c,m)$ of colors in $\cP^C$ is a subgroup of the
automorphism group of the collection $c$ in $\cR$. Thus, if $\cR$ is
$\Sigma$-cofibrant, $\cR^C$ is $\Sigma$-cofibrant as well.
This proves the assertion a).

{\sl Assertion (b).}
Here we follow, with minor amendments, Lurie \cite{L.HA}, 4.1.4.13
\footnote{We use  the term "$\colim$-adapted" instead of Lurie's "good".}.
A morphism of simplicial complexes $X\rTo Y$ will be called {\sl $\colim$-adapted},
if
\begin{itemize}
\item $X,Y$ are $\colim$-adapted.
\item For each $n\in\Delta^\op$ the map $X_n\rTo Y_n$ is a cofibration in $C(k)$.
\item The map $\colim X\rTo\colim Y$ is a cofibration in $C(k)$.
\end{itemize}

The collection of $\colim$-adapted morphisms satisfies the following properties 
(proven in \cite{L.HA}, 4.1.4.13).

\begin{itemize}
\item Let $X$ be a cofibrant complex. Then the constant simplicial complex defined by $X$ is $\colim$-adapted.
\item Let $f,g$ be two $\colim$-adapted maps. Then $f\wedge g$ defined as in (\ref{eq:wedge}) is $\colim$-adapted.
\item All cofibrations in $C(k)^{\Delta^\op}$ are $\colim$-adapted.
\item The collection of $\colim$-adapted morphisms is closed under transfinite composition.
\item Base change: if $f:X\to Y$ and $Z$ in a pushout diagram
\begin{equation}
\begin{diagram}
X & \rTo^f & Y \\
\dTo & & \dTo \\
Z&\rTo^{f'}& T
\end{diagram}
\end{equation}
are $\colim$-adapted, then $f'$ is also $\colim$-adapted.
\end{itemize}

Let us now prove assertion (b). 

Obviously, a retract of a $\colim$-adapted simplicial complex is $\colim$-adapted,
so we may assume that the cofibrant algebra $A$ is a colimit of a transfinite 
sequence of cofibrations $A=\colim A_\alpha$ where $A^0$ is the initial 
$\cR^{\Delta^\op}$-algebra,
$A^\alpha\simeq\colim_{\beta<\alpha}A_\beta$ for $\alpha$ limit ordinal, and such that
$A_{\alpha+1}$ is obtained from $A_\alpha$ via a pushout diagram
\begin{equation}
\begin{diagram}
\F_{\cR^{\Delta^\op}}(V) & \rTo & \F_{\cR^{\Delta^\op}}(W) \\
\dTo && \dTo \\
A_\alpha & \rTo & A_{\alpha+1}
\end{diagram},
\end{equation}
where $V\to W$ is a cofibration of cofibrant objects in $C(k)^{[\cR]\times\Delta^\op}$.

According to \ref{prp:pullback}, $A_{\alpha+1}$, considered as an object of $C(k)^{[\cR]\times\Delta^\op}$, is obtained as a colimit of arrows obtained by pushforward along
the maps (\ref{eq:BB}). Since the components of $A_{c(i)}$ are cofibrant complexes
and $\cR$ is homotopically sound, the properties of $\colim$-adapted maps listed above imply that the map $\G(A_\alpha)\to \G(A_{\alpha+1})$ is $\colim$-adapted.
This proves the assertion.

\end{proof}

\subsubsection{}
To prove assertion 4, we will need a minor 
generalization  of \cite{L.HA}, 3.1.3.11, describing free algebras generated 
by a collection of objects corresponding to different colors.

Let $\cO^\otimes$ be an $\infty$-operad and $a:I\to[\cO]$ a collection of 
objects in $\cO$. Let $\Theta\subset\Fin_*$ denote the subcategory defined 
by the inert arrows in $\Fin_*$, so that $\Theta$ is the ``trivial 
$\infty$-operad". We identify $\Theta^I$ with the category whose objects are
finite sets over $I$ and whose morphisms are inert partial maps over $I$. There is 
an (essentially unique) extension of the map $a:I\to\cO$ to a 
map $\theta:\Theta^I\to \cO^\otimes$ of $\infty$-operads.

Let $q:\cC^\otimes\to\cO^\otimes$ be an $\cO$-monoidal $\infty$-category. 
A collection of objects $X=\{X_i\in\cC_i\},\ i\in I$ 
defines (essentially uniquely) a $\Theta^I$-algebra $\bar X$ in $\cC$ such that 
$q\circ\bar X=\theta:\Theta^I\to\cO^\otimes$. 
Let $F\in\Alg_\cO(\cC)$. The lemma 
below allows one to check whether a given morphism $\bar X\to\theta^*(F)$ exhibits 
$F$ as free $\cO$-algebra generated by the collection $\{X_i\},\ i\in I$.
 
 Denote $\Theta^I_\iso$ the maximal subgroupoid of $\Theta^I$. In other words, 
 this is the groupoid of finite sets over $I$.
For each $y\in \cO$ we define a Kan simplicial set $\cP_{I,y}$ 
as the full sub($\infty$-)category of 
$\Theta^I_\iso\times_{\cO^\otimes}\cO^\otimes_{/y}$ spanned by  
the objects whose component in $\cO^\otimes_{/y}$ is given by an active arrow.

One has a canonical map $h:\cP_{I,y}\times\Delta^1\rTo \cO^\otimes$ defined as follows.

Its restriction $h_0$ to $\cP_{I,y}\times\{0\}$ is the composition $\cP_{I,y}\to \Theta^I_\iso\to\cO^\otimes$
whereas the restriction $h_1$ to $\cP_{I,y}\times\{1\}$ carries  
everything to $\{y\}\in\cO$. In general,
if $\pi$ is a $k$-simplex of $\cP_{I,y}$ and $\sigma_i$ is a $k$-simplex of $\Delta^1$ having $i$ times 
value $1$ ($i=0,\ldots,k+1$), then the image of $(\pi,\sigma_i)$ is defined by the formula
\begin{equation}
h(\pi,\sigma_i)=\left\{
\begin{array}{ll}
d_{k+1}\tau, & i=0\\
s^{i-1}_{k-i+1}d^i_{k-i+1}\tau, & i>0
\end{array}
\right.
\end{equation}
 where $\tau$ is the $k+1$-simplex of $\cO^\otimes$ defining
the projection of $\pi$ to $\cO^\otimes_{/y}$.
 
Now, the map $q:\cC^\otimes\to\cO^\otimes$ being cocartesian
fibration,
the map $h:\cP_{I,y}\times\Delta^1\rTo\cO^\otimes$ can be lifted
to a map $H:\cP_{I,y}\times\Delta^1\rTo\cC^\otimes$
so that the restriction $H_0$ to $\cP_{I,y}\times\{0\}$ is 
the composition
$$ \cP_{I,y}\rTo\Theta^I_\iso\rTo^{\bar X}\cC^\otimes$$
and such that for each $\pi\in \cP_{I,y}$ the arrows
$H(\pi,\Delta^1)$ is cocartesian.

This yields a map $H_1: \cP_{I,y}\rTo\cC_y$.

\begin{Dfn}
The colimit of $H_1$ constructed above, if exists, is denoted 
$\Sym_\cO(\bar X)_y$. 
\end{Dfn}

Let now, $\bar X$ be as above, and let $A$ be a a $\cO$-algebra
in $\cC$. A choice of a map of $\Theta^I$-algebras 
$f:\bar X\to\theta^*(A)$,
which is given essentially by a collection of maps
$f_i:X_i\rTo A(i),\ i\in I$,
defines  a canonical collection of maps
\begin{equation}
F_y:\Sym_\cO(\bar X)_y\rTo A_y
\end{equation}
as follows. The map $f:\bar X\to\theta^*(A)$ induces a map
$$H_0\rTo A\circ h_0:\cP_{I,y}\rTo\cC^\otimes.$$

By construction of $\Sym_\cO(\bar X)_y$, one obtains a canonical map
$$H\rTo A\circ h:\cP_{I,y}\times\Delta^1\rTo\cC^\otimes$$
whose restriction to $\cP_{I,y}\times\{1\}$ yields a map 
$F_y:\Sym_\cO(\bar X)_y\rTo A_y$.

The following lemma is a straightforward generalization of \cite{L.HA}, 3.1.3.11.

\begin{lem}\label{lem:31311}
A collection of maps $f_i:X_i\to \theta^*(A)_i,\ i\in I$, exhibits $A$ as a free 
algebra generated by $X$ iff for all $y\in\cO$
the natural map
$$F_y:\Sym_{\cO}(\bar X)_y\rTo A_y$$
is an equivalence.
\end{lem}
\qed

We will now apply the above lemma to prove Assertion 4.
In our context $\cO^\otimes$ is the $\infty$-operad constructed from a 
topological operad $\cO$. 
We put $I=[\cO]$ and we represent the collection of $X_i\in\QC(k)$ by their cofibrant representatives $Y_i$.
Let $\F_\cR(Y)$ be the free $\cR$-algebra on $Y=\{Y_i\}_{i\in[\cO]}$.
We denote $\F$ to be the $\cO$-algebra in $\QC(k)$ corresponding to $\F_\cR(Y)$ as explained
in \ref{ss:phi}. 
We have canonical  maps $Y_i\to\theta^*(\F)_i$, so we can 
apply the above lemma.

It remains to check that the maps 
$F_y:\Sym_{\cO}(\bar X)_y\to \F_y$
are equivalences.

Recall that $\Theta^I_\iso$ identifies with the groupoid of collections $\Fin/[\cO]$
used in the description of the (classical) free algebra, see~\ref{thm:adjalg}.
One has a canonical projection $\cP_{I,y}\to N\Fin/[\cO]$ and $n$-simplices
of $\cP_{I,y}$ over $\sigma:c_0\to\ldots\to c_n$ in $N\Fin/[\cO]$ correspond to $n$-simplices of $\cO(c_0,y)$.

 The map $\Sym_\cO(\bar X)_y\rTo\F_y$ is constructed as follows.
The map $H_1:\cP_{I,y}\rTo\QC$ is the composition
\begin{equation}\label{eq:H1}
\cP_{I,y}\rTo N\Fin/[\cO]\rTo \QC,
\end{equation}
where the second arrow carries a collection $c:J\to[\cO]$ to $\otimes_{j\in J}Y_{c(j)}$. 

The canonical maps $Y_i\rTo\theta^*(\F)_i$ allow one to extend $H_1$ to a functor
$$H_1^\triangleright:\cP^\triangleright_{I,y}\rTo\QC$$
carrying the vertex on the left to $\F_y=\colim\Free(Y)_y,$
where the functor $\Free(Y)_y:\Fin/[\cO]\rTo C(k)$ is given by the formula $\Free(Y)_y(c)=\cR(c,y)\otimes
\bigotimes_{j\in J}Y_{c(j)}$, see formula~(\ref{eq:freealg}). 

The extension $H_1^\triangleright$ is constructed as follows. One chooses a 
section $s$ of the projection $N_*(\cR(c,y))\rTo N_*(C_*(\cO(c,y)))$. We will denote by the
same letter the composition $s:\cO(c,y)\rTo N_*(C_*(\cO(c,y)))\rTo N_*(\cR(c,y))$.
Now, given an $n$-simplex $(\sigma,\tau)$ of $\cP_{I,y}$, where $\sigma:c_0\to\ldots\to c_n$
in $N\Fin/[\cO]$ and $\tau\in\cO(c_0,y)_n$, one extends it with the map
$$ C_*(\Delta^n)\otimes\bigotimes_{j\in J_0}Y_{c_0(j)}\rTo \F_y$$
defined by the map $C_*(\Delta^n)\rTo \cR(c_0,y)$ determined by $s(\tau)$. 
It remains to check that 
$H_1^\triangleright:\cP^\triangleright_{I,y}\rTo\QC(k)$
is a colimit diagram.

In fact, since the functor $H_1$ factors through $N\Fin/[\cO]$,
see (\ref{eq:H1}), $\colim\ H_1$ can be calculated as the colimit
of the left Kan extension of $H_1$ via the projection 
$$ \cP_{I,y}\rTo N\Fin/[\cO].$$
This functor is a Kan fibration, so by \cite{L.T}, 4.3.3.1 the left
Kan extension of $H_1$ is precisely the functor $\mathfrak{F}(Y)_y$. 
Since $\cR$ is $\Sigma$-cofibrant, its (naive) colimit calculates
as well the required homotopy colimit.

\qed

\subsection{Algebras over a PROP}

Theorem~\ref{thm:rect-alg} allows one to get a certain rectification result for algebras  over a topological PROP, or, more generally, over any SM topological category.

A topological SM category $\cP$ determines a SM $\infty$-category $\cP^\otimes$,
see~\ref{ss:introrec},
and the $\infty$-category of algebras $\Fun^\otimes(\cP^\otimes,\QC(k)^\otimes)$.

We are going to give a ``classical" description of this notion of $\infty$-algebra.

\begin{dfn}\label{dfn:prop-alg}
Let $\cR$ be a symmetric monoidal dg category. A homotopy $\cR$-algebra in $C(k)$ is a lax SM functor
$A:\cP\rTo C(k)$ such that the natural map
$$ A(x)\otimes A(y)\rTo A(x\otimes y)$$
induces a quasiisomorphism $A(x)\otimes^\Left A(y)\rTo A(x\otimes y)$ for all $x,y\in \cP$.
\end{dfn}

We denote by $\cP^o$ the dg operad defined by $\cP$. A lax SM functor $\cP\to C(k)$ is just a strict $\cP^o$-algebra in $C(k)$. 

We need a minor generalization of the above definition. A dg operad $\cR$ will be called
{\sl weak SM category}  if the corresponding operad enriched over
the derived category of $k$ is an (enriched) SM category.

The only example of weak SM category we need is the following.

\begin{lem}
Let $\cP$ be a SM dg category and let $\cR\to \cP^o$ be a homotopically
sound replacement of dg operads.
Then $\cR$ is a weak SM category.
\end{lem}
\qed

If $\cR$ is a weak SM category and $x,y\in [\cR]$ two objects, tensor product $x\otimes y$ is
defined uniquely up to equivalence. Then Definition \ref{dfn:prop-alg}
is applicable also for weak SM categories. We will repeat it once more.

\begin{dfn}\label{dfn:prop-alg-weak}
Let $\cR$ be a weak SM dg category. A homotopy $\cR$-algebra in $C(k)$ is an algebra over the operad $\cR$ such that the natural map
$$ A(x)\otimes A(y)\rTo A(x\otimes y)$$
induces a quasiisomorphism $A(x)\otimes^\Left A(y)\rTo A(x\otimes y)$ for all $x,y\in [\cR]$.
\end{dfn}

Let $\cP$ be a topological SM category and let $\cR\rTo C_*(\cP^o,k)$ be a homotopically sound replacement.

The $\infty$-category of homotopy
$\cR$-algebras is defined as the full subcategory of $N(\Alg^\st_{\cR}(C(k)))$ consisting of 
homotopy $\cR$-algebras. It can be otherwise described as the DK localization of the category of cofibrant homotopy $\cR$-algebras, with respect to weak equivalences.

Theorem~\ref{thm:rect-alg} immediately implies the following result.

\begin{crl}\label{crl:rect-prop-alg}
Let $\cP$ be a topological SM category, and let $\cR\to C_*(\cP,k)$ be a homotopically sound replacement of $C_*(\cP,k)$ considered as 
an operad.
Let $\cP^\otimes$ be the SM $\infty$-category defined by $\cP$. Then the 
 equivalence of $\infty$-categories
$$ \Phi:N(\Alg^\st_{\cR}(C(k))\rTo\Alg_{P^\otimes}(\QC(k))$$
induces an equivalence of the subcategory of homotopy $\cR$-algebras with 
$\Fun^\otimes(\cP^\otimes,\QC(k)^\otimes)$.
\end{crl}
\qed

\section{Modules}

In this section we deduce from Theorem~\ref{thm:rect-alg} the rectification for modules over
operad algebras. Our definition of module over an $\cO$-operad algebra
 is very straightforward.
For any $\infty$-operad $\cO$ we define a new $\infty$-operad denoted $M\cO$ such that
algebras over $M\cO$ are pairs $(A,M)$ where $A$ is an $\cO$-algebra and $M$ is an $A$-module, see Subsection~\ref{ss:MO}. 

Theorem~\ref{thm:rect-alg} implies the main result of this section
Theorem~\ref{crl:str-mod} saying that the $\infty$-category of modules over an operad algebra can be described as the infinity category underlying the corresponding model category. The precise formulation is given in \ref{crl:str-mod}. The proof is based on the
result on localization of families of $\infty$-categories given in \cite{H.L}, Sect. 2. 

In his foundational book~\cite{L.HA} J.~Lurie gives another definition which, under some restrictions, yields for any $\cO$-algebra $A$ an $\cO$-monoidal category of $A$-modules. In Appendix B we show that our definition is equivalent to 
the one suggested by Lurie, with discarded $\cO$-monoidal structure.

\subsection{Classical setting}\label{ss:cl-mod}

Let $\cO$ be a topological operad.
We define a new topological operad $\Mu\cO$ as follows. We double each color, defining
$$ [\Mu\cO]=[\cO]\times\{a,m\}.$$
A collection of colors in $\Mu\cO$ is given by a pair of maps
$\wt c=(c,X_c)$ where $c:I\to [\cO],\ X_c:I\to\{a,m\}$.

The space of operations $\Mu\cO(\wt c,\wt d)$ for $\wt c:I\to[\Mu\cO],\ 
\wt d\in[\Mu\cO]$
is nonempty in two cases described below.
\begin{itemize}
\item $X_c(i)=a$ for all $i\in I$, $X_d=a$. 
\item $X_c(i)=m$ for precisely one $i\in I$, $X_d=m$.
\end{itemize}
In these cases $\Mu\cO(\wt c,\wt d)=\cO(c,d)$.

Note that the same construction makes sense for operads enriched
over any SM category.

One has a map $\cO\to\Mu\cO$
carrying any $c\in[\cO]$ to $(c,a)$. In the opposite direction,
a map $\Mu\cO\to\cO$ erases the $X$-marking of a color.

Algebras over $\Mu\cO$ are pairs $(A,M)$ where $A$ is an $\cO$-algebra 
and $M$ is a $A$-module.

\subsection{}
\label{ss:MO}
A similar construction makes sense in the context of $\infty$-operads.
Given an $\infty$-operad $\cO^\otimes$ we define a new operad $\Mu\cO^\otimes$
by the formula
$$ \Mu\cO^\otimes=\CM^\otimes\times_{N\Fin_*}\cO^\otimes,$$
where $\CM$ is the two-color operad governing pairs $(A,M)$ with $A$
a commutative algebra and $M$ an $A$-module, and $\CM^\otimes$ is the 
corresponding $\infty$-operad.

\begin{dfn}\label{dfn:module}
Let $\cO^\otimes$ be an $\infty$-operad and $\cC^\otimes$
a SM $\infty$-category.
Let $A\in\Alg_{\cO}(\cC)$. The $\infty$-category
$\Mod_A^{\cO}(\cC)$ is defined as the fiber product
$$\Mod_A^{\cO}(\cC)=
\Alg_{\Mu\cO}(\cC)\times_{\Alg_{\cO}(\cC)}\{A\}
$$
\end{dfn}

Let $\cO$ be a topological operad and let $\cR\rTo C_*(\cO)$ be a homotopically sound replacement.
  
We are not sure that $\Mu\cR$ is always homotopically sound.
The following, however, is very easy. 

\begin{lem}\label{lem:ssplitmod}
Let $\cR$ be a $\Sigma$-split operad in $C(k)$. Then $\Mu\cR$ is also
$\Sigma$-split. If $\cR$ is $\Sigma$-cofibrant, then  $\Mu\cR$ is $\Sigma$-cofibrant. 
\end{lem}
\begin{proof}
If the collection of maps $t_\theta:\cR(c,d)\to \cR(c,d)$, 
$\theta:I\to\langle|I|\rangle$, provides a $\Sigma$-splitting for $\cR$, the same maps provide a $\Sigma$-splitting for nonzero components of $\Mu\cR$ in the notation of  \ref{ss:cl-mod} and \ref{ss:sigmasplit}. The second claim is also obvious.
\end{proof}

In any case, we can choose a homotopically sound replacement 
$$ \cM\rTo C_*(\Mu\cO)$$
and define $\cR$ as the full suboperad spanned by the original colors
of $\cO$. Then by \ref{thm:modelham} $\cR$ is a homotopicaly sound replacement of $C_*(\cO)$.

\begin{thm}\label{crl:str-mod}
Let $\cO$ be a topological operad. Let $\cM$ be a homotopically sound replacement
of $C_*(\Mu\cO)$  and let $\cR$ be the full suboperad of $\cM$ spanned by the original colors
of $\cO$. Let $A$ be an $\cR$-algebra in $C(k)$. 
There is a canonical equivalence
\begin{equation}\label{eq:mod}
N(\Mod^\cR_A(C(k)))\to\Mod_A^{\cO}(\QC(k)),
\end{equation}
where we denote by the same letter $A$ the corresponding $\cO$-algebra in $\QC(k)$.
\end{thm}
\begin{proof}
 According to Theorem~\ref{thm:rect-alg} one has a commutative diagram
of $\infty$-categories whose horizontal maps are equivalences.
\begin{equation}
\begin{diagram}
N(\Alg^\st_{\cM}(C(k))) & \rTo & \Alg_{\Mu\cO}(\QC(k)) \\
\dTo^{\phi^\st} & & \dTo^{\phi} \\
N(\Alg^\st_{\cR}(C(k))) & \rTo & \Alg_{\cO}(\QC(k))
\end{diagram}
\end{equation}
This yields an equivalence of the homotopy fibers of the vertical maps.
It remains to identify the map of homotopy fibers  with the
map (\ref{eq:mod}).

The forgetful functor 
\begin{equation}
\Alg^\st_{\cM}(C(k))\rTo \Alg^\st_{\cR}(C(k))
\end{equation}
inducing the left vertical arrow $\phi^\st$, preserves cofibrant algebras.
This map, restricted to the subcategories spanned by the cofibrant objects,
and considered as marked $\infty$-categories (quasiisomorphisms being 
the marked arrows), is a marked cocartesian fibration in the sense of
Definition 2.1.1,~\cite{H.L}: this is a cocartesian fibration of categories
(a map $f:A\to A'$ of algebras gives rise to base change map $f^*:\Mod_A\to\Mod_{A'}$), the base change preserves quasiisomorphisms of cofibrant modules,
and weak equivalence of cofibrant algebras gives rise to equivalence of
the corresponding categories of modules). 
Then, according to Proposition 2.1.4,~\cite{H.L}, the (homotopy) fibers of $\phi^\st$ 
identify with the DK localizations of the fiber of the functor
\begin{equation}
\phi^\st:\Alg^\st_\cM(C(k))^c\rTo\Alg^\st_\cR(C(k))^c
\end{equation}
at a cofibrant algebra $A$.

If we had $\cM=\Mu\cR$, the fiber would be precisely the category of cofibrant $A$-modules. In general one has to add a few lines.

\

We will now present a mini-theory, generalizing to algebras over
colored operads the notion of enveloping algebra.

Let $\cR$ be a full suboperad of an operad $\cM$ such that 
the following ``linearity'' condition holds.

{\sl
Let $c:I\to[\cM]$ and $d\in[\cM]$ satisfy the condition 
$\cM(c,d)\ne 0.$ Then either both $d$ and the image of $c$ belong to $[\cR]$, or $d\not\in[\cR]$ and there is precisely one $i\in I$ such that $c(i)\not\in[\cR]$.
 }
 Given a pair of operads $\cM\supset\cR$ satisfying the linearity condition, and an $\cR$-algebra $A$ in $C(k)$, we can look at the
fiber of the functor $\Alg_\cM(C(k))\to\Alg_\cR (C(k))$ at $A$
as the category of generalized $A$-modules. In case $\cM=\Mu\cR$
this fiber {\sl is} the category of $A$-modules.  Keeping in mind this
analogy, we will denote it $\Mod_A^{\cM}$. 

\

The category $\Mod_A^\cM$ can be easily described as the category of 
representations (that is, $C(k)$-enriched functors to $C(k)$) of a
category enriched over $C(k)$ which we will call {\sl the enveloping
category of $A$} and will denote $U^\cM_\cR(A)$. Its objects are the
 elements of $[\cM]-[\cR]$. The complex of maps from $x$ to $y$,
$x,y\in[\cM]-[\cR]$, is a certain colimit of tensor products of $\cM(c,d)$ along a category of marked trees.

In case $\cR$ is a colorless operad and $\cM=\Mu\cR$, we recover the
classical notion of universal enveloping algebra. In case $\cR$ is a full suboperad of two operads $\cM$ and $\cM'$, a map of operads
$f:\cM\to\cM'$ is said to be over $\cR$ if $f|_\cR=\id$. If $A$ is a cofibrant $\cR$-algebra and $f:\cM\to\cM'$ is a quasiisomorphism
of operads over $\cR$, the induced map
$$ U_\cR^\cM(A)\rTo U_\cR^{\cM'}(A)$$
is an equivalence of dg categories.
This can be checked precisely as in the colorless case, see~\cite{haha}, 5.3.3. 

We can now complete our proof applying the above claim to $\cR$, $\cM$
as in the theorem and $\cM'=\Mu\cR$.
\end{proof}

\appendix

\section{Symmetric monoidal adjunction}
\label{sec:SMA}

Let $F:C\rlarrows D:G$ be an adjoint pair of symmetric monoidal categories, so that $F$
is a symmetric monoidal functor. Then it is easy to see that $G$ is automatically lax symmetric monoidal. 

In this subsection we study the above phenomenon and its generalizations
in the context of $\infty$-categories.

\subsection{Fibrations in $\Cat_\infty$}

We are going to use the notions of left or cocartesian fibration applied to arrows of $\Cat_\infty$ rather than of $\sSet$ as in \cite{L.T}, 2.1 and 2.4. Here are the appropriate definitions.
 
A map $f:X\to Y$ in $\Cat_\infty$ is called a left fibration if the diagram below
defined by $f$
\begin{equation}
\begin{diagram}
X^{\Delta^1} & \rTo & Y^{\Delta^1} \\
\dTo & & \dTo \\
X^{\{0\}} & \rTo^f & Y^{\{0\}}
\end{diagram}
\end{equation}
is cartesian.  Equivalently, this means that a map $f$ is equivalent to one represented by a left fibration in $\sSet$.

Similarly, one defines a cocartesian fibration in $\Cat_\infty$ as a map equivalent
to one represented by a cocartesian fibration in $\sSet$. Let $X,\ Y$ be $\infty$-categories. A map $f:X\to Y$ in $\sSet$ represents a cocartesian fibration in $\Cat_\infty$ if it can be embedded into a homotopy commutative triangle of   
$\infty$-categories
\begin{equation}
\begin{diagram}
X && \rTo^h && Z\\
& \rdTo^f && \ldTo^g \\
&& Y &&
\end{diagram},
\end{equation}
where $g$ is a cocartesian fibration in $\sSet$ and $h$ is an equivalence of $\infty$-categories.

Given an $\infty$-category $C$, we denote as $\Coc_C$ the
subcategory of $(\Cat_\infty)_{/C}$ spanned by cocartesian fibrations,
with arrows preserving the cocartesian liftings. The category
$\LEFT_C$ is the full subcategory of $(\Cat_\infty)_{/C}$ spanned 
by the left fibrations. One also defines the $\infty$-category $\LEFT$ as the full
subcategory of $\Fun(\Delta^1,\Cat_\infty)$ spanned by the left fibrations.
The $\infty$-cateory $\LEFT_C$ is the fiber at $C$ of the cartesian fibration 
$e_1:\Fun(\Delta^1,\Cat_\infty)\to \Cat_\infty$ assigning to each arrow in $\Cat_\infty$ its target.

\subsection{SM Grothendieck construction}

Recall that for an $\infty$-category $C$ there is an equivalence
\begin{equation}\label{eq:coc}
\Coc_C\stackrel{\sim}{\rTo} \Fun(C,\Cat_\infty).
\end{equation}
In this subsection we will describe symmetric monoidal versions of this correspondence.

A map of SM $\infty$-categories is called SM cocartesian fibration,
see \cite{L.HA}, 2.1.2.13,
if it is presented by a cocartesian fibration of the corresponding $\infty$-categories over $N\Fin_*$.
In the following proposition $\Coc^\SM_{C^\otimes}$ denotes the category of SM cocartesian fibrations over $C^\otimes$.
 
\begin{prp}
Let $C^\otimes$ be a SM $\infty$-category. There is an equivalence 
\begin{equation}
\Coc^\SM_{C^\otimes}\stackrel{\sim}{\rTo} \Fun^\lax(C^\otimes,\Cat_\infty),
\end{equation}
compatible with the equivalence (\ref{eq:coc}).
\end{prp} 
\begin{proof}
Since $\Cat_\infty$ is cartesian closed, the right hand side identifies with the
full subcategory of $\Fun(C^\otimes,\Cat_\infty)$ spanned by the functors $F:C^\otimes\rTo\Cat_\infty$  which are {\sl lax cartesian structures} in the sense of \cite{L.HA}, 2.4.1.1: any object $X=X_1\oplus\ldots\oplus X_n$ with $X_i\in C$ exhibits 
$F(X)$ as the product of $F(X_i)$.

Now the claim follows from the equivalence of two definitions of operad cocartesian fibration, see~\cite{L.HA}, 2.1.2.12.
\end{proof}

The following result is an immediate consequence of the above.
\begin{crl}\label{crl:lefteq}
Let $C^\otimes$ be a SM $\infty$-category. There is an equivalence 
\begin{equation}
\LEFT^\SM_{C^\otimes}\stackrel{\sim}{\rTo}  \Fun^\lax(C^\otimes,\cS)
\end{equation}
between the $\infty$-category of left fibrations $M^\otimes\rTo C^\otimes$
{\sl which are SM functors} and lax SM functors $C^\otimes\rTo\cS$,
compatible with the Grothendieck construction.
\end{crl}

The $\infty$-category $\LEFT^\SM_{C^\otimes}$ has a simple interpretation in terms of $\LEFT$.
The latter $\infty$-category has a cartesian SM structure. One has
\begin{lem}\label{lem:sm-left}
Let $C^\otimes$ be a SM $\infty$-category. The obvious functor
\begin{equation}
\LEFT^\SM_{C^\otimes}\rTo \Alg_\Com(\LEFT)
\end{equation}
identifies the left-hand side with the fiber of the forgetful functor
\begin{equation*}
\Alg_\Com(\LEFT)\rTo\Cat^\SM_\infty
\end{equation*}
at $C^\otimes$.
\end{lem}
\begin{proof}
We have to check that if a SM functor $f:D^\otimes\rTo C^\otimes$ induces
a left fibration $D\rTo C$ of the respective $\infty$-categories, then $f$ itself is a left fibration. Applying Proposition 2.4.2.11 of \cite{L.T}, we get that $f$ is a locally cocartesian fibration. Moreover, the same proposition gives a description of locally cocartesian arrows:
these are arrows $\alpha:d\rTo d'$ embeddable into a commutative triangle
\begin{equation}
\begin{diagram}
d & & \rTo^\alpha & & d' \\
& \rdTo^\beta & & \ruTo^\gamma & \\
&& d^{\prime\prime} &&
\end{diagram},
\end{equation}
where $\beta$ is $\pi_D$-cocartesian for $\pi:D^\otimes\rTo N\Fin_*$, and $\pi(\gamma)=\id$.
Obviously, these are {\sl all} arrows in $D^\otimes$. Therefore, $f$ is a left fibration.
\end{proof}

\subsection{SM pairings}

A pairing of $\infty$-categories is a pair of maps $C\lTo M\rTo D$, such that
the induced map $M\rTo C\times D$ is a left fibration
\footnote{Note: Lurie \cite{L.X}, 3.1 and 4.2, uses right fibrations instead.}. 

The $\infty$-category of pairings,
$\Pair$, is defined as the full subcategory of the category $\Fun(\Lambda^2_0,\Cat_\infty)$,
spanned by the diagrams giving rise to a left fibration.

Equivalently, $\Pair$ can be defined by a cartesian diagram
\begin{equation}
\begin{diagram}
\Pair & \rTo  & \LEFT \\
\dTo && \dTo \\
\Cat_\infty\times\Cat_\infty & \rTo^\times & \Cat_\infty
\end{diagram}.
\end{equation}

A pairing is uniquely defined, up to a usual ambiguity, by a corresponding functor to the category of spaces $C\times D\rTo\cS$.

The forgetful functor $\Pair\rTo \Cat_\infty\times\Cat_\infty$ induces a functor
$$\Alg_\Com(\Pair)\rTo\Cat^\SM_\infty\times\Cat^\SM_\infty.$$

% This is a cartesian fibration.
% The proof uses \cite{L.T}, 2.4.7.11, 2.4.7.5 

For a pair $(C,D)$ of SM $\infty$-categories we denote as $\Alg_\Com(\Pair)_{(C,D)}$ the fiber
of this functor at $(C,D)$.

According to Corollary~\ref{crl:lefteq} and Lemma~\ref{lem:sm-left}, one has an equivalence 
\begin{equation}\label{eq:compairs}
 \Alg_\Com(\Pair)_{(C,D)}\stackrel{\sim}{\rTo}\Fun^\lax(C\times D,\cS).
\end{equation}

\subsection{Representability}

A pairing $(p,q):M\rTo C\times D$ is called {\sl left-representable} if for any $x\in C$ the fiber $p^{-1}(x)$ has an initial object.  

We define $\Pair^\ell$ as the full subcategory of $\Pair$ spanned
by the left-representable pairings
\footnote{This differs from the category of left-representable pairings
considered in \cite{L.X}, 4.2.7 where the arrows are required to be left-representable.}. 

A left pairing $p:M\to C\times D$ corresponds to a functor $C\times D\rTo\cS$
which can be equivalently converted into a functor $\wt p:C\rTo P(D^\op)$.
\footnote{Here $P(X)$ denotes the presehaves on $X$.}
Then $p$ is left representable iff $\wt p$ factors through Yoneda embedding
$D^\op\rTo P(D^\op)$. More precisely, one has the following.

\begin{lem}\label{lem:leftpair}
The equivalence 
$$\Pair_{(C,D)}\stackrel{\sim}{\rTo}\Fun(C\times D,\cS)$$
identifies the full subcategory $\Pair^\ell_{(C,D)}$ on the left with $\Fun(C,D^\op)$
on the right.
\end{lem}
\begin{proof}
We have to verify that the natural map $\Fun(C,D^\op)\rTo\Fun(C\times D,\cS)$
is fully faithful. This is the composition of the equivalence
$$\Fun(C,P(D^\op))\rTo \Fun(C\times D,\cS)$$
with the map
$$\Fun(C,D^\op)\rTo\Fun(C,P(D^\op))$$
which is fully faithful by Yoneda lemma.
\end{proof}

$\Pair^\ell$ is closed under direct products.
Therefore, commutative algebras in $\Pair^\ell$ form a full subcategory of $\Alg_\Com(\Pair)$.

Our aim is to prove the following SM version of Lemma~\ref{lem:leftpair}.

\begin{prp}\label{prp:leftpair-SM}
The equivalence (\ref{eq:compairs})
$$\Alg_\Com(\Pair)_{(C,D)}\stackrel{\sim}{\rTo}\Fun^\lax(C\times D,\cS)$$
identifies the full subcategory $\Alg_\Com(\Pair^\ell)_{(C,D)}$ on the left with $\Fun^\lax(C,D^\op)$ on the right.
\end{prp}

The proof is given in \ref{sss:op}--\ref{sss:prppf} below.

\subsubsection{Opposite $S$-family}
\label{sss:op}

Given a cocartesian fibration $p:C\rTo S$ corresponding to 
$p':C\rTo\Cat_\infty$, we define a cocartesian fibration 
$p^\circ:C^\circ\rTo S$
as the one corresponding to the composition of $p'$ with the functor 
$X\mapsto X^\op$.

More explicitly, if $p$ is presented by a cocartesian fibration $C\rTo S$ in $\sSet$,
the cocartesian fibration $C^\circ\rTo S$ is defined by the formula
\begin{multline}\label{eq:opc-hom}
 \Hom_S(T,C^\circ)= \Hom^\ell(T\times_SC,\cS):=\\
 \{f\in\Hom(T\times_SC,\cS)|\forall t\in T\ f_t:\{t\}\times_SC\to\cS\textrm{ is corepresentable }\}.
\end{multline} 

Note that the formula~(\ref{eq:opc-hom}) yields immediately
 
\begin{equation}\label{eq:funrel}
\Fun_S(T,C^\circ)=\Fun^\ell(T\times_SC,\cS),
\end{equation} 
where $\Fun^\ell(T\times_SC,\cS)$ is defined as
$$ \Fun^\ell(T\times_SC,\cS)_n=\Hom^\ell((T\times\Delta^n)\times_SC,\cS).$$
This is a full subcategory of $\Fun(T\times_SC,\cS)$ consisting of {\sl left-representable functors}.

\subsubsection{Proof of \ref{prp:leftpair-SM}}
\label{sss:prppf}

We will apply the formula (\ref{eq:funrel}) to $S:=N\Fin_*$,
$T:=C^\otimes$, $C:=D^\otimes$.
The right-hand side of (\ref{eq:funrel}) contains the
full subcategory 
$$\Fun^{\ell,\lax}(C^\otimes\times_{N\Fun_*}D^\otimes,\cS)$$
spanned by the left-representable lax cartesian structures
on $C^\otimes\times_{N\Fun_*}D^\otimes$ in the sense of \cite{L.HA}, 2.4.4.1.

We will now check that the corresponding full subcategory of 
$\Fun_{N\Fin_*}(C^\otimes,D^{\circ\otimes})$ coincides with 
$\Fun^\lax(C,D^\circ)$.

In fact, let $f:C^\otimes\times_{N\Fin_*}D^\otimes\rTo\cS$ be lax cartesian and left-representable. Recall that $f$ is lax cartesian if for $c=\oplus_{i=1}^n c_i,\ d=\oplus_{i=1}^nd_i$ the natural map
$f(c,d)\rTo\prod f(c_i,d_i)$ defined by the inert maps $\rho^i:\langle n\rangle\to\langle 1\rangle$, is an equivalence. In this case left representability can be checked at $\langle 1\rangle\in N\Fin_*$ only;
it will follow automatically for all $\langle n\rangle\in N\Fin_*$.

Left representability of $f$ yields a map $\wt f:C^\otimes\rTo D^{\circ\otimes}$ over $N\Fin_*$. If $c=\otimes c_i$ and $d=\wt f(c)$, we have immediately
$d=\oplus\wt f(c_i)$ which is equivalent to preservation of inert arrows.

This proves Proposition \ref{prp:leftpair-SM}.

\subsection{A generalization}
\label{ss:conclusion}

Proposition~\ref{prp:leftpair-SM} has an obvious (and important) generalization 
to multi-variable adjunction.

Let $K$ be a subset of $\{1,\ldots,n\}$.

A functor $F:C_1\times\ldots\times C_n\rTo\cS$ will be called $K$-representable if
the corresponding functor $F':\prod_{i\in K}C_i\rTo P(\prod_{i\not\in K}C_i^\op)$
factors through the fully faithful embedding 
$$\prod_{i\not\in K}C_i^\op\rTo P(\prod_{i\not\in K}C_i^\op).$$

Let now $C_i$ be symmetric monoidal.
The claim below directly follows from Proposition~\ref{prp:leftpair-SM}.

\begin{crl}\label{crl:laxadjunction}
There is an equivalence between the following $\infty$-categories.
\begin{itemize}
\item[1.] $\Fun^\lax(\prod_{i\in K}C_i,\prod_{i\not\in K}C_i^\circ)$,
\item[2.] The full subcategory of $\Fun^\lax(\prod_{i=1}^nC_i,\cS)$ spanned by 
the lax functors which are $K$-representable, once the SM structure is discarded.
\end{itemize}
\end{crl}

\begin{exm}\label{exm:1}
In particular, if a SM functor $f:C\rTo D$ admits a right adjoint as a functor between 
$\infty$-categories, its right adjoint has a canonical lax SM structure, see also
\cite{H.L}, 3.1.1.
\end{exm}

\begin{exm}\label{exm:2}
A symmetric monoidal functor $f:C\rTo D$ determines on $D$ a $C$-tensored structure, defined by a functor
$$ C\times D\rTo D, \quad (x,y)\mapsto f(x)\otimes y$$
which is also symmetric monoidal. This implies that, if $f$ induces also a $C$-cotensored
structure on $D$
$$ C^\op\times D\rTo D,$$
it is automatically lax symmetric monoidal.
\end{exm}

%%%%%%%%%%%%%%%%%%%%%%%%%%%

\section{Comparison of two notions of module}
\label{app:2mod}

\subsection{}
In this appendix we assume that the operad $\cO^\otimes$ is unital
(see \cite{L.HA}, 2.3.1), that is that the space $\Map(\emptyset,x)$
is contractible for any $x\in\cO$. Here $\emptyset$ belongs to the contractible space $\cO^\otimes_{\langle 0\rangle}$.

Denote $\cS_\cO$ the full subcategory of $\Fun(\Delta^1,\cO^\otimes)$
spanned by the semi-inert arrows (see~\cite{L.HA}, 3.3.1) $x\to y$ in $\cO^\otimes$ with $p(x)=\langle 1\rangle\in\Fin_*$
\footnote{In the notation of Lurie~\cite{L.HA}, 3.3.2.1, $\cS_\cO$ is the fiber of the composition $p\circ e_0:\cK_\cO^\otimes\to\cO^\otimes\to N\Fin_*$ at $\langle 1\rangle$.}.

The maps $s,t:\cS_\cO\rTo\cO^\otimes$ assign to an edge its source and its target, respectively.

\begin{lem}\label{lem:Scat}
The map $t:\cS_\cO\rTo\cO^\otimes$ is a categorical fibration. 
\end{lem}

\begin{proof}
The map is the composition 
$\cS_\cO\to\Fun(\Delta^1,\cO^\otimes)
\rTo^{e_1}\cO^\otimes$.  The second map is a cartesian fibration by \cite{L.T}, 2.4.7.11 and 2.4.7.5. In particular, it is a categorical 
fibration. The first map is an embedding as a full subcategory,
so is an inner fibration. Now Joyal's criterion \cite{L.T}, 2.4.6.5
immediately shows it is also a categorical fibration.

\end{proof}

\ 

An edge $\alpha$ in  $\cS_\cO$ will be called inert if both $s(\alpha)$ and $t(\alpha)$ are inert edges in $\cO^\otimes$. Note that $s(\alpha)$ has to be an equivalence since it lives over an inert
endomorphism of $\langle 1\rangle\in\Fin_*$ which has to be identity. 

The assignment $\cO^\otimes\mapsto\cS_\cO$ is functorial.
One can identify $\cS_\Com$ with a subpreoperad of $\CM^{\otimes\natural}$
via the map
\begin{equation}
\iota:\cS_\Com\rTo\CM^{\otimes\natural}
\end{equation}
carrying an arrow $\alpha:\langle 1\rangle\to I_*$ to the 
characteristic function $h:I\to\{a,m\}$ of the image of $\alpha$: $h$ 
has value $m$ on the image and $a$ otherwise.

\

The commutative diagram
\begin{equation}
\begin{diagram}
\cS_\cO & \rTo & \cS_\Com & \rTo & \CM^\otimes \\
\dTo^t & & \dTo^t & & \dTo \\
\cO^\otimes & \rTo & N\Fin_* & \rEqual& N\Fin_*
\end{diagram}
\end{equation}
defines the maps 
\begin{equation}\label{eq:piiota}
\cS_\cO\rTo^\pi \cS_\Com\times_{N\Fin_*}\cO^\otimes\rTo^\iota \CM^\otimes
\times_{N\Fin_*}\cO^\otimes.
\end{equation}
 
The following result shows that Definition~\ref{dfn:module} of the category
of modules over an $\infty$-operad algebra is equivalent to the one given by 
Lurie in \cite{L.HA}, 3.3.3.8.

In particular, Corollary~\ref{crl:str-mod} is applicable to Lurie modules (with $\cO$-monoidal structure discarded).

\begin{prp}\label{prp:apprx}
Assume that $\cO^\otimes$ is unital and $\cO$ is Kan (this is so if $\cO^\otimes$ 
is coherent \cite{L.HA}, 3.3.1.9). Then the maps $\pi$ and $\iota$ 
in (\ref{eq:piiota}) are weak equivalences in $\Pop_\infty$.
\end{prp}

The proof of the proposition is given in \ref{sss:start}--\ref{sss:end}.

We will prove that the $\iota$ and $\iota\circ\pi$  are weak equivalences in 
$\Pop_\infty$.

This will be done using Lurie's notion of {\sl approximation of operads.}
We check that both $\iota$ and $\iota\circ\pi$ are approximations
in the sense of \cite{L.HA}, 2.3.3.6.

Then, using \cite{L.HA},  2.3.3.23(1), we deduce that the maps $\iota$ and 
$\iota\circ\pi$ are weak equivalences.

\subsection{Proof of \ref{prp:apprx}}
 
\subsubsection{$\iota$ is an approximation} 
\label{sss:start}
The map $\iota$ is obtained from the embedding $\iota_\Com:\cS_\Com\to\CM^\otimes$
by a base change along fibration $p:\cO^\otimes\to N\Fin_*$. Therefore, by  
Remark 2.3.3.9 of \cite{L.HA}, in order to prove $\iota$
is an approximation, it suffices to check that the map $\iota_\Com$ 
is an approximation.
$\cS_\Com$ is a full subcategory of $\CM^\otimes$ spanned by the objects having 
at most one appearance of $m$.
Thus, if an arrow $\alpha:x\to y$ in $\CM^\otimes$ is inert and $x\in\cS_\Com$ 
then $y\in\cS_\Com$. Similarly, if $\alpha$ is active and $y\in\cS_\Com$ then 
$x\in\cS_\Com$. This implies $\iota_\Com$ is an approximation. 

\

We will now prove that $\iota\circ\pi$ is also an approximation.
This is done in \ref{sss:apprx-2-begin}--\ref{sss:apprx-2-end}.

\subsubsection{$\iota\circ\pi$ is a categorical fibration}
 \label{sss:apprx-2-begin}

Let $A\to B$ be a trivial cofibration of simplicial sets in the Joyal model 
structure.
Given a pair of compatible maps $a:A\to\cS_\cO$ and 
$b:B\to\CM^\otimes\times_{N\Fin_*}\cO^\otimes$, we have to find a lifting
$c:B\to\cS_\cO$ making two triangles commutative.

We proceed as follows. Look at the commutative square with vertical arrows
$A\to B$ and $\cS_\cO\to\cO^\otimes$, as shown in the diagram.
\begin{equation}\label{eq:ABSO}
\begin{diagram}
A & \rTo & \cS_\cO &&\\
\dTo &\ruDashto^c & \dTo_{\iota\circ\pi} & \rdTo^t &\\
B & \rTo & \CM^\otimes\times_{N\Fin_*}\cO^\otimes & \rTo & \cO^\otimes
\end{diagram}
\end{equation}
 
By Lemma~\ref{lem:Scat} there is a lifting $c:B\to \cS_\cO$ making two 
triangles $AB\cS_\cO$ and $B\cS_\cO\cO^\otimes$ commutative. 
We claim $c$ makes also commutative the triangle we need. In fact, the 
commutative diagram
\begin{equation}
\begin{diagram}
A & \rTo &\CM^\otimes \\
\dTo & & \dTo \\
B & \rTo & N\Fin_*
\end{diagram}
\end{equation}
has a unique lifting (as $\CM^\otimes$ and $N\Fin_*$ are both nerves of 
categories, and the functor between them admits lifting of isomorphisms).
This implies that the lifting $c$ automatically satisfies the required 
property.

\subsubsection{Property (1) of \cite{L.HA}, 2.3.3.6.}
 
If $a:x\to y$ is an object in $\cS_\cO$ and $\bar b:p(y)\to \bar z$ is an inert
arrow in $\Fin_*$, we can lift
$\bar b$ to an inert edge $b:y\to z$ in $\cO^\otimes$ (so that $p(b)=\bar b$). 
One has a triangle $x\rTo^ay\rTo^b z$ in $\cO^\otimes$ that determines an edge in 
$\cS_\cO$. 
Its image is obviously inert in $\CM^\otimes\times_{N\Fin_*}\cO^\otimes$.

\subsubsection{Property (2) of \cite{L.HA}, 2.3.3.6.}

Here  the assumptions of \ref{prp:apprx} on $\cO^\otimes$ will be used.

Let $a:x\to y$ be an object of $\cS_\cO$ and let $(\alpha,y)\in\CM^\otimes\times_{N\Fin_*}\cO^\otimes$
denote the image $\iota\circ\pi(a)$.

An active edge $\beta:(\gamma,z)\rTo(\alpha,y)$ in $\CM^\otimes\times_{N\Fin_*}\cO^\otimes$ 
is uniquely determined by an active edge $b:z\to y$  in $\cO^\otimes$, together with an element
$\gamma:p(x)\to p(z)$ in $\cS_\Com$ such that $p(a)=p(b)\circ\gamma$.

We have to find a cartesian lifting for $\beta$, that is a 2-simplex in 
$\cO^\otimes$
\begin{equation}
\begin{diagram}
& & z & & \\
& \ruTo^c & & \rdTo^b &\\
x && \rTo^a && y
\end{diagram},
\end{equation}
such that $p(c)=\gamma$ and satisfying a certain universal property.

In case $a$ is null, the map $\gamma$ is null and we choose $c$ to be the null 
map from $x$ to $z$. The required 2-simplex is now essentially unique as $\Map(0,y)$ is
contractible.

In case $a$ is not null let $y=\oplus y_i$ with $y_i\in\cO$, $z=\oplus z_i$
with $z_i\in\cO^\otimes$,\\
 $b=\oplus b_i:z_i\to y_i$.
Let, furthermore, $a$ be defined by an equivalence $a_k:x\to y_k$. Write 
$z_k=\oplus z^j$ where $z^j\in\cO$.
The map $\gamma:p(x)\to p(z)$ factors
through $\gamma':p(x)\to p(z_k)$ which singles out an element $j$ in $p(z_k)$. 
This allows one to produce
an arrow $z^j\to y_i$ in $\cO$ obtained from $b_i:z_i\to y_i$ by precomposing with units. It should
be equivalence as $\cO$ is Kan.
Choose a two-simplex $x\to z^j\to y_k$ with the described above edges 
$x\to y_k$ and $z^j\to y_k$. Adding to it an essentially unique triangle
$$ 0\rTo z\ominus z^j\rTo^0 y\ominus y_k, $$
we get a triangle with required properties
\footnote{Of course $y\ominus y_k=\oplus_{i\ne k}y_i$ and similarly for $z\ominus z^j$.}.

We claim that the edge in $\cS_\cO$ defined by the above two-simplex, is a 
$\iota\circ\pi$-cartesian lifting of $\beta:(\gamma,z)\rTo(p(a),y)$.

The map $f:\cS_\cO\rTo\CM^\otimes\times_{N\Fin_*}\cO^\otimes$ is a categorical 
fibration, so the criterion 2.4.4.3 of \cite{L.T} can be applied.

We have to check that for any $d:s\to w$ in $\cS_\cO$ the following homotopy
commutative diagram
\begin{equation}\label{eq:cart}
\begin{diagram}
\Map_{\cS_\cO}(d,c) & \rTo & \Map_{\cS_\cO}(d,a) \\
\dTo & & \dTo \\
\Map_{\CM^\otimes\times_{N\Fin_*}\cO^\otimes}(f(d),f(c)) & \rTo & 
\Map_{\CM^\otimes\times_{N\Fin_*}\cO^\otimes}(f(d),f(a)) 
\end{diagram}
\end{equation}
is homotopy cartesian\footnote{To make this formulation precise, one has to
replace the map spaces with their explicit representatives by Kan simplicial sets,
so that the diagram (\ref{eq:cart}) is commutative.}.

Since $\cS_\cO$ is a full subcategory of $\Fun(\Delta^1,\cO^\otimes)$, we can replace $\Map_{\cS_\cO}$
in the above diagram with $\Map_{\Fun(\Delta^1,\cO^\otimes)}$. Here the following easy lemma 
is very convenient.

\begin{lem}\label{lem:funmap}
Let $\cC$ be an $\infty$-category and $\cD=\Fun(\Delta^1,\cC)$.
Let $a:x\to y$ and $a':x'\to y'$ be two objects in $\cD$. Then one has a 
homotopy cartesian diagram
\begin{equation}
\begin{diagram}\label{eq:fund}
\Map_\cD(a,a') & \rTo & \Map_\cC(y,y') \\
\dTo & & \dTo \\
\Map_\cC(x,x') & \rTo & \Map_\cC(x,y')
\end{diagram}
\end{equation}
\end{lem}
\begin{proof}
The formulation of the lemma is imprecise: the diagram (\ref{eq:fund}) has to be replaced with an
explicit commutative diagram of spaces, similarly to the  diagram (\ref{eq:cart}).
It is described below.
By definition, $\Map_\cC(x,y')$ is realized as the fiber
of the categorical fibration
\footnote{See Lurie, \cite{L.T}, 4.2.1, this is the realization via
``alternative join"}
$$ \Fun(\Delta^1,\cC)\rTo\Fun(\partial\Delta^1,\cC)=\cC^2$$
at the point $(x,y')$. 
Similarly, $\Map_\cD(a,a')$ is realized as the fiber of 
$$ \Fun(\Delta^1\times\Delta^1,\cC)=\Fun(\Delta^1,\cD)\to \Fun(\partial\Delta^1,\cD)=
\Fun(\Delta^1,\cC)^2$$
at the point $(a,a')$.

Furthermore, we replace the space $\Map_\cC(y,y')$ with $\Map'_\cC(y,y')$
defined as the fiber of the map
$$ \Fun(\Delta^2,\cC)\rTo\Fun(\Delta^1,\cC)\times \cC,$$
induced by $\partial^2:\Delta^1\to \Delta^2$ and $\partial^0\partial^1:\Delta^0\to\Delta^2$,
at the point $(a,y')$. Similarly, we replace $\Map_\cC(x,x')$ with $\Map'_\cC(x,x')$
defined as the fiber of the map
$$\Fun(\Delta^2,\cC)\rTo\cC\times\Fun(\Delta^1,\cC),$$
induced by $\partial^0:\Delta^1\to \Delta^2$ and $\partial^1\partial^2:\Delta^0\to\Delta^2$,
at the point $(x,a')$.

The canonical maps $\Map'_\cC(y,y')\to\Map_\cC(y,y')$ and 
$\Map'_\cC(x,x')\to\Map_\cC(x,x')$ are trivial Kan fibrations
as they can be obtained by base change from the trivial fibration
\begin{equation}
\Fun(\Delta^2,\cC)\rTo\Fun(\Lambda^2_1,\cC).
\end{equation}
This, in particular, proves that all $\Map'$-spaces are
Kan.

The commutative square 
\begin{equation}
\begin{diagram}\label{eq:fund'}
\Map_\cD(a,a') & \rTo & \Map'_\cC(y,y') \\
\dTo & & \dTo \\
\Map'_\cC(x,x') & \rTo & \Map_\cC(x,y')
\end{diagram}
\end{equation}
replacing the diagram (\ref{eq:fund}), is now obtained from the following commutative cube

\begin{equation}
\begin{diagram}
\Fun(\Delta^1\times\Delta^1,\cC) & & \rTo^{} & & \Fun(\Delta^2,\cC) & & \\
& \rdTo_{} & & & \vLine^{} & \rdTo_{} & \\
\dTo^{} & & \Fun(\Delta^2,\cC) & \rTo^{} &\HonV&&\Fun(\Delta^1,\cC)\\
& & \dTo^{} & & \dTo & & \\
\Fun(\Delta^1,\cC)^2 & \hLine & \VonH & \rTo^{} &\Fun(\Delta^1,\cC)\times\cC & & \dTo_{} \\
& \rdTo_{} & & & & \rdTo_{} & \\
& & \cC\times\Fun(\Delta^1,\cC) & & \rTo^{} & & \cC^2 \\
\end{diagram}
\end{equation}
as the fiber along the vertical edges at 
$(a,a')\in\Fun(\Delta^1,\cC)^2$ and at its images.

Here the upper face is defined by the presentation of $\Delta^1\times\Delta^1$ as a union of two 2-simplices glued along
an edge. The upper and the lower faces are cartesian and homotopy cartesian. Therefore, the commutative square (\ref{eq:fund'}) of the fibers is also homotopy cartesian.
\end{proof}

\subsubsection{End of the proof of Property (2)}
\label{sss:apprx-2-end}

Recall that we have to verify that the diagram (\ref{eq:cart})
is homotopy cartesian.

Note that
\begin{eqnarray}\label{eq:bottoml}
\Map_{\CM^\otimes\times_{N\Fin_*}\cO^\otimes}(f(d),f(c))=\hspace{6cm}\\
\nonumber\Map_{\cO^\otimes}(w,z)\times_{\Hom_{\Fin_*}(p(w),p(z))}\Hom_{\cS_\Com}(p(d),p(c))=\\
\nonumber\Map_{\cO^\otimes}(w,z)\times_{\Hom_{\Fin_*}(p(s),p(z))}\Hom_{\Fin_*}(p(s),p(x)),
\end{eqnarray}
and similarly
\begin{eqnarray}\label{eq:bottomr}
\Map_{\CM^\otimes\times_{N\Fin_*}\cO^\otimes}(f(d),f(a))=\hspace{6cm}\\
\nonumber\Map_{\cO^\otimes}(w,y)\times_{\Hom_{\Fin_*}(p(s),p(y))}\Hom_{\Fin_*}(p(s),p(x)).
\end{eqnarray}

Let us first replace the diagram (\ref{eq:cart}) with a commutative diagram so that the claim
become formally meaningful. Similarly to what we did in the proof of Lemma~\ref{lem:funmap},
we replace in (\ref{eq:cart}) $\Map_{\cS_\cO}(d,c)$ and 
$$\Map_{\CM^\otimes\times_{N\Fin_*}\cO^\otimes}(f(d),f(c))=
\Map_{\cO^\otimes}(w,z)\times_{\Hom_{\Fin_*}(p(s),p(z))}\Hom_{\Fin_*}(p(s),p(x))$$
with homotopy equivalent versions, $\Map'_{\cS_\cO}(d,c)$ and 
$\Map'_{\cO^\otimes}(w,z)\times_{\Hom_{\Fin_*}(p(s),p(z))}\Hom_{\Fin_*}(p(s),p(x))$
where $\Map'_{\cS_\cO}(d,c)$ is the fiber of the map 
\begin{equation}\label{eq:prime-1}
\Fun(\Delta^2,\cS_\cO)\rTo \cS_\cO\times\Fun(\Delta^1,\cS_\cO)
\end{equation}
at $(d,\beta:c\to a)$ whereas $\Map'_{\cO^\otimes}(w,z)$ is the fiber of
\begin{equation}\label{eq:prime-2}
\Fun(\Delta^2,\cO^\otimes)\rTo\cO^\otimes\times\Fun(\Delta^1,\cO^\otimes)
\end{equation}
at $(w,b:z\to y)$.
The maps (\ref{eq:prime-1}) and (\ref{eq:prime-2}) are both induced by the embedding
$\Delta^0\sqcup\Delta^1\rTo\Delta^2$ defined by $\partial^1\partial^2$ and by $\partial^0$.

The diagram (\ref{eq:cart})\footnote{rotated by 90$^\circ$ to fit on the page} is now replaced with a commutative diagram
\begin{equation}\label{eq:cart'}
\begin{diagram}
\Map_{\cS_\cO}(d,c) & \rTo^{\pi_c} & \Map_{\cO^\otimes}(w,z)\times_{\Hom_{\Fin_*}(p(s),p(z))}
\Hom_{\Fin_*}(p(s),p(x))\\ 
\uTo^\simeq & & \uTo^\simeq \\
\Map'_{\cS_\cO}(d,c) & \rTo^{\pi'_c} &\Map'_{\cO^\otimes}(w,z)\times_{\Hom_{\Fin_*}(p(s),p(z))} 
\Hom_{\Fin_*}(p(s),p(x))\\ 
\dTo & & \dTo \\
\Map_{\cS_\cO}(d,a) & \rTo^{\pi_a} & \Map_{\cO^\otimes}(w,y)\times_{\Hom_{\Fin_*}(p(s),p(y))} 
\Hom_{\Fin_*}(p(s),p(x)) 
\end{diagram}
\end{equation}
The upwards arrows are weak equivalences; we will prove that the lower commutative square
is homotopy cartesian, considering 
separately the cases $a=0$ and $a\ne 0$ \footnote{we denote by $0$ a null map
which exists and is essentially unique for any choice of source and target.}.

In case $a$ is non-null, the edge $c$ is also non-null. In this case we will 
check that the horizontal arrows $\pi_c$ and $\pi_a$ of (\ref{eq:cart'}) are equivalences. 
The case $a=0=c$ will be verified separately.

{\bf Case $a\ne 0$.}  

We will prove that if $d:s\to w$ and $a:x\to y$ are in $\cS_\cO$ so that 
$a\ne 0$, then the natural map
$$\pi_a:\Map_{\cS_\cO}(d,a)\rTo
\Map_{\cO^\otimes}(w,y)\times_{\Hom_{\Fin_*}(p(s),p(y))}\Hom_{\Fin_*}(p(s),p(x))$$
is an equivalence.

The proof goes as follows. 
We replace $\Map_{\cS_\cO}$ with $\Map_{\Fun(\Delta^1,\cO)}$ and use Lemma~\ref{lem:funmap}
to express it as a fiber product. We have to check 
therefore that the map
\begin{eqnarray}
\Map'_{\cO^\otimes}(w,y)\times_{\Map_{\cO^\otimes}(s,y)}\Map'_{\cO^\otimes}(s,x)\rTo\hspace{3cm} \\
\nonumber\Map_{\cO^\otimes}(w,y)\times_{\Hom_{\Fin_*}(p(s),p(y))}\Hom_{\Fin_*}(p(s),p(x))
\end{eqnarray}
is an equivalence, where the notation for $\Map'$ is as in Lemma~\ref{lem:funmap}.

 Since the map
$$\Map'_{\cO^\otimes}(w,y)\rTo\Map_{\cO^\otimes}(w,y)$$
is a trivial fibration, it is sufficient to check that the diagram 
\begin{equation}
\begin{diagram}
\Map'_\cO(s,x) & \rTo & \Map_{\cO^\otimes}(s,y) \\
\dTo & & \dTo \\
\Hom_{\Fin_*}(p(s),p(x))& \rTo & \Hom_{\Fin_*}(p(s),p(y))
\end{diagram}
\end{equation}
is cartesian if $a\ne 0$. In other words, we have to check that for any map
$e:p(s)\to p(x)$ (there are two such maps as $p(s)=p(x)=\langle 1\rangle$)
the fiber of the left vertical map at $e$ is equivalent to the fiber of the
right vertical map at $p(a)\circ e$. If $e=0$, both fibers are contractible.
Otherwise write $y=\oplus y_i$ with $y_i\in\cO$, so that $a$ is determined by an equivalence
$a_k:x\to y_k$ for some $k$. Then the fibers are equivalent respectively to $\Map'_\cO(s,x)$
and to $\Map_\cO(s,y_k)$, that is, equivalent to each other.

{\bf Case $a=0$.}

% Sept 24 version

We will show that the lower commutative square in (\ref{eq:cart'}) is equivalent to the
following diagram
\begin{equation}\label{eq:cart0}
\begin{diagram}
\Map_\cO(s,x) \times F_c & \rTo & \Hom_{\Fin_*}(p(s),p(x))\times F_c \\
\dTo & & \dTo \\
\Map_\cO(s,x) \times F_a & \rTo & \Hom_{\Fin_*}(p(s),p(x))\times F_a
\end{diagram}
\end{equation}
for appropriately chosen $F_c$ and $F_a$. This will imply the claim.

We proceed as follows. 

Define $\bar F_a$ as the fiber of the map $\Map_{\cO^\otimes}(w,y)\rTo\Hom_{\Fin_*}(p(s),p(y))$
at zero. Then the target of $\pi_a$ in the diagram (\ref{eq:cart'}) identifies with
$\bar F_a\times\Hom_{\Fin_*}(p(s),p(x))$. Similarly, we define $\bar F_c$ as the fiber
of the map $\Map'_{\cO^\otimes}(w,z)\rTo\Hom_{\Fin_*}(p(s),p(z))$
at zero. This will identify the target of $\pi'_c$ in the diagram (\ref{eq:cart'}) with
$\bar F_c\times\Hom_{\Fin_*}(p(s),p(x))$.

The projections $s,t:\cS_\cO\to\cO^\otimes$ yield the maps
\begin{equation}
s_a:\Map_{\cS_\cO}(d,a)\rTo\Map_\cO(s,x),\quad
t_a:\Map_{\cS_\cO}(d,a)\rTo\Map'_{\cO^\otimes}(w,y),
\end{equation}
where $\Map'_{\cO^\otimes}(w,y)$ is defined as the fiber of the map
$$ \Fun(\Delta^2,\cO^\otimes)\rTo \Fun(\Delta^1,\cO^\otimes)\times\cO^\otimes$$
defined by $\partial^2:\Delta^1\to\Delta^2$ and $\partial^0\partial^0:\Delta^0\to\Delta^2$, 
at $(d,y)$. The composition 
$$ 
\Map_{\cS_\cO}(d,a)\rTo^{t_a}\Map'_{\cO^\otimes}(w,y)\rTo\Map_{\cO^\otimes}(s,y)
$$ is zero. We can therefore define $F_a$ as the fiber of
\begin{equation}
\Map'_{\cO^\otimes}(w,y)\rTo\Map_{\cO^\otimes}(s,y)
\end{equation}
at $0$ (at the contractible space of null maps), and get a canonical map
\begin{equation}
\Map_{\cS_\cO}(d,a)\rTo \Map_\cO(s,x)\times F_a.
\end{equation}
One easily sees this is an equivalence. 
Note that one has a canonical map $F_a\to \bar F_a$ which is a trivial Kan fibration.

Similarly, one has a pair of maps
\begin{equation}
s_c:\Map'_{\cS_\cO}(d,c)\rTo\Map_\cO(s,x),\quad
t_c:\Map_{\cS_\cO}(d,c)\rTo\Map^{\prime\prime}_{\cO^\otimes}(w,z),
\end{equation}
with $\Map^{\prime\prime}_{\cO^\otimes}(w,z)$ defined as the fiber of the map
\begin{equation}
\Fun(\Delta^3,\cO^\otimes)\rTo\Fun(\Delta^1,\cO^\otimes)^2
\end{equation}
given by $\partial^2\partial^3:\Delta^1\to\Delta^3,\quad \partial^0\partial^0:\Delta^1\to\Delta^3$, at $(d,b)$.

Once more the composition
$$ 
\Map'_{\cS_\cO}(d,c)\rTo^{t_c}\Map^{\prime\prime}_{\cO^\otimes}(w,z)\rTo\Map_{\cO^\otimes}(s,z)
$$ 
is zero, so we define $F_c$ as the fiber of
\begin{equation}
\Map^{\prime\prime}_{\cO^\otimes}(w,z)\rTo\Map_{\cO^\otimes}(s,z)
\end{equation}
at $0$   and get a canonical map
\begin{equation}
\Map'_{\cS_\cO}(d,c)\rTo \Map_\cO(s,x)\times F_c.
\end{equation}
It is also an equivalence and the map $F_c\to\bar F_c$ is a trivial Kan fibration.

We are done.

\

\subsubsection{}
\label{sss:end}
Having checked that $\iota$ and $\iota\circ\pi$ are approximations, we 
can now deduce from Theorem 2.3.3.23(2) of \cite{L.HA} that both $\iota$ and $\iota\circ\pi$
are weak equivalences.


\begin{thebibliography}{MMMM}
\bibitem[BM1]{BM0} C.~Berger, I.~Moerdijk, Axiomatic homotopy theory
for operads, Comm. Math. Helv., 78(2003), 805--831.
\bibitem[BM2]{BM} C.~Berger, I.~Moerdijk, Resolution of coloured oprerads and rectification of homotopy algebras, Categories in algebra, geometry and mathematical physics, 31–-58, Contemp. Math., 431, Amer. Math. Soc., Providence, RI, 2007,  arXiv:math/0512576.
\bibitem[BM3]{BM3} C.~Berger, I.~Moerdijk, On the derived category of an algebra over an operad. Georgian Math. J. 16 (2009), no. 1, 13-–28, arXiv:0801.2664.
\bibitem[C]{C} K.~Costello, Topological conformal field theories and Calabi-Yau categories,
Adv. Math. 210 (2007), no. 1, 165–-214,  arXiv:math/0412149.  
\bibitem[H]{haha} V.~Hinich, Homological algebra of homotopy algebras,  Comm. Algebra 25 (1997), no. 10, 3291–-3323,  arXiv:q-alg/9702015.
\bibitem[H.err]{haha-err} V.~Hinich, Erratum to ``Homological algebra of homotopy algebras", preprint  arXiv:math/0309453.
\bibitem[H.L]{H.L} V.~Hinich, Dwyer-Kan localization revisited,
to appear in Homology, Homotopy and Applications,  arXiv: 1311.4128
\bibitem[H.V]{virtual}V.~Hinich, Virtual operad algebras and realization of homotopy types, 
J. Pure Appl. Algebra 159 (2001), no. 2--3, 173-185, arXiv:math/9907115.
\bibitem[H.DSA]{DSA} V. Hinich, Deformations of sheaves of algebras, Adv. Math. 195(2005), no. 1, 102--164, arXiv:math/0310116.
\bibitem[Hir]{hirsch} P.~Hirschhorn, Model categories and their localizations,  Mathematical Surveys and Monographs, 99. AMS, Providence, RI, 2003. xvi+457 pp.
%\bibitem[Lyu]{lyu} V.~Lyubashenko, Homotopy unital 
%A$_\infty$-algebras. J. Algebra 329 (2011), 190--212.
\bibitem[L.HA]{L.HA} J.~Lurie, Higher algebra, preprint August 3, 2012,
available at
http://www.math.harvard.edu/\textasciitilde{}lurie/papers/HigherAlgebra.pdf.
\bibitem[L.T]{L.T} J.~Lurie, Higher topos theory, Annals of 
Mathematics Studies, 170. Princeton University Press, Princeton, NJ,
 2009. xviii+925 pp, also available at
 http://www.math.harvard.edu/\textasciitilde{}lurie/papers/croppedtopoi.pdf.
\bibitem[L.X]{L.X} J.~Lurie, Derived algebraic geometry, X, available
 at http://www.math.harvard.edu/\textasciitilde{}lurie/papers/DAG-X.pdf.
\bibitem[Mu]{Mu} F.~Muro, Homotopy theory of nonsymmetric operads, Algebr. Geom. Topol. 11 (2011), no. 3, 1541–-1599, arXiv:1101.1634.
\bibitem[SS00]{SS00}S.~Schwede, B.~Shipley, Algebras and modules in monoidal model categories. Proc. London Math. Soc. (3) 80 (2000), no. 2, 491–-511,
 arXiv:math/9801082.
\bibitem[SS03]{SS} S.~Schwede, B.~Shipley, Equivalence of monoidal model categories, Algebraic and geometric topology, 3(2003), 287--334,
arXiv:math/0209342
\bibitem[T1]{T1} G.~Tabuada, Th\'eorie homotopique des DG-cat\'egories,
th\`ese Univ. Paris 7, arXiv:0710.4303.
\bibitem[T2]{T2} G.~Tabuada, Differential graded versus simplicial
categories, Topology Appl. 157 (2010), no. 3, 563--593,
arXiv:0711.3845.


\end{thebibliography}
\end{document}